\documentclass[a4paper,oneside,11pt]{article}%
\usepackage{makeidx}
\usepackage[english]{babel}
\usepackage{amsmath}
\usepackage{amsfonts}
\usepackage{amssymb}
\usepackage{stmaryrd}
\usepackage{graphicx}
\usepackage{mathrsfs}
\usepackage[colorlinks,linkcolor=red,anchorcolor=blue,citecolor=blue,urlcolor=blue]{hyperref}
\usepackage[symbol*,ragged]{footmisc}

\allowdisplaybreaks[4]
\providecommand{\U}[1]{\protect\rule{.1in}{.1in}}
\providecommand{\U}[1]{\protect \rule{.1in}{.1in}}

\newtheorem{theorem}{Theorem}[section]

\newtheorem{lemma}[theorem]{Lemma}

\newenvironment{proof}[1][Proof]{\noindent \textbf{#1.} }{\  \rule{0.5em}{0.5em}}
\numberwithin{equation}{section}

\begin{document}

\title{On the distribution of $\alpha p^2$ modulo one \\ with prime $p$ of a special form }

\author{Fei Xue\footnotemark[1] \,\,\,\,\,  \& \,\,\, Jinjiang Li\footnotemark[2]\,\,\,\,\,\,\,  \& \,\,\,
        Min Zhang\footnotemark[3]
                    \vspace*{-4mm} \\
     $\textrm{\small Department of Mathematics, China University of Mining and Technology\footnotemark[1]\,\,\,\footnotemark[2]}$
                    \vspace*{-4mm} \\
     \small  Beijing 100083, P. R. China
                     \vspace*{-4mm}  \\
     $\textrm{\small School of Applied Science, Beijing Information Science and Technology University\footnotemark[3]}$
                    \vspace*{-4mm}  \\
     \small  Beijing 100192, P. R. China  \vspace*{-4mm}  \\}

\footnotetext[2]{Corresponding author. \\
    \quad\,\, \textit{ E-mail addresses}:
     \href{mailto:fei.xue.math@gmail.com}{fei.xue.math@gmail.com} (F. Xue),
     \href{mailto:jinjiang.li.math@gmail.com}{jinjiang.li.math@gmail.com} (J. Li),\\
     \quad \qquad\qquad\quad\quad\quad\quad \,\,\,\,\,
     \href{mailto:min.zhang.math@gmail.com}{min.zhang.math@gmail.com} (M. Zhang).  }

\date{}
\maketitle

{\textbf{Abstract}}: Let $\mathcal{P}_r$ denote an almost--prime with at most $r$ prime factors, counted
according to multiplicity. In this paper, it is proved that for
$\alpha\in\mathbb{R}\backslash\mathbb{Q},\,\beta\in\mathbb{R}$ and $0<\theta<10/1561$, there exist infinitely
many primes $p$, such that $\|\alpha p^2+\beta\|<p^{-\theta}$ and $p+2=\mathcal{P}_4$, which constitutes an improvement
upon the previous result.


{\textbf{MR(2020) Subject Classification}}: 11J71, 11N36, 11L07, 11L20

\section{Introduction and main result}

Let $\mathcal{P}_r$ denote an almost--prime with at most $r$ prime factors, counted according to
multiplicity. The famous prime twins conjecture states that there exist infinitely many
primes $p$ such that $p+2$ is a prime, too. Up to now, this conjecture is still open,
but many approximations about this conjecture were established. One of the most interesting results is due to
Chen \cite{Chen-Jingrun-1973}, who showed, in 1973, that there exist infinitely many primes $p$ such that
$p+2=\mathcal{P}_2$.

In 1981, Heath--Brown \cite{Heath-Brown-1981} showed that there exist infinitely many arithmetic
progressions of four different terms, three of which are primes, and the fourth is $\mathcal{P}_2$.
In 2006, Green and Tao \cite{Green-Tao-2006} established that there exist infinitely many arithmetic
progressions consisting of three different primes $p_1<p_2<p_3$ such that $p_j+2=\mathcal{P}_2$ for each $j=1,2,3$.
Later, in 2008, Green and Tao \cite{Green-Tao-2008} showed that for any $k\geqslant3$ there exist infinitely many arithmetic progressions
consisting of $k$ different primes $p_1<p_2<\cdots<p_k$ such that $p_j+2=\mathcal{P}_2$ for each $j=1,2,\dots,k$.

Suppose that there is a problem including primes and let $r\geqslant 2$ be an integer. Having in mind Chen's result,
one may consider the problem with primes $p$, such that $p+2=\mathcal{P}_r$. Many authors investigated several
kinds of problems of this type , such as Peneva and Tolev \cite{Peneva-Tolev-1998}, Peneva \cite{Peneva-2000},
Tolev \cite{Tolev-1999,Tolev-2000-1,Tolev-2000-2}, etc.

Let $\alpha$ be a irrational real number and $\|x\|$ denote the distance from $x$ to the nearest
integer. Earlier work about the distribution of the fractional parts of the sequence $\{\alpha p\}$
was first considered by Vinogradov \cite{Vinogradov-book}, who showed that for any real number $\beta$, there
are infinitely many primes $p$ such that for $\theta=1/5-\varepsilon$, then
\begin{equation}\label{linear-case}
  \|\alpha p+\beta\|<p^{-\theta},
\end{equation}
where and below $\varepsilon$ denotes arbitrarily small positive number. After that, the first improvement on (\ref{linear-case}) was due to Vaughan \cite{Vaughan-1977}, who obtained $\theta=1/4$ in (\ref{linear-case}), and who also required an additional factor $(\log p)^8$ on the right hand side of (\ref{linear-case}).
Since then, many authors improved the upper bound of the exponent $\theta$, such as Harman \cite{Harman-1983-1,Harman-1996}, Jia \cite{Jia-1993,Jia-2000},
Heath--Brown and Jia \cite{Heath-Brown-Jia-2002}, etc. So far the best result is given by Matom\"{a}ki \cite{Matomaki-2009} with $\theta=1/3-\varepsilon$. Moreover, it seems very natural to consider the sequence $\{\alpha p^k\}$ for $k\geqslant2$, where $p$ denotes a prime variable. Also, many authors studied the fractional parts of the sequence $\{\alpha p^k\}$ for $k\geqslant2$, such as Baker and Harman \cite{Baker-Harman-1991}, Harman \cite{Harman-1983-2}, Wong \cite{Wong-1997}, etc.

In 2010, Todorova and Tolev \cite{Todorova-Tolev-2010} considered the distribution of $\alpha p$ modulo one with
primes of the form specified above, and showed that for $\theta=1/100$, there are infinitely many solutions in
primes $p$ to (\ref{linear-case}) such that $p+2=\mathcal{P}_4$. Later, Matom\"{a}ki \cite{Matomaki-2009-JNT}
showed that this result actually holds with $p+2=\mathcal{P}_2$ and $\theta=1/1000$. After that, Shi \cite{Shi-2012}
continue to improve the result of Matom\"{a}ki \cite{Matomaki-2009-JNT}, and showed that there are infinitely
many solutions in primes $p$ to (\ref{linear-case}) such that $p+2=\mathcal{P}_2$ and $\theta=3/200$.

Moreover, for the case $k=2$, Shi and Wu \cite{Shi-Wu-2013} established the result that there
exist infinitely many primes $p$, which satisfy $\|\alpha p^2+\beta\|<p^{-\theta}$, such that
$p+2=\mathcal{P}_4$ and $\theta=2/375-\varepsilon$.

In this paper, we shall continue to improve the result of Shi and Wu \cite{Shi-Wu-2013}
and establish the following theorem.

\begin{theorem}\label{Theorem}
  Suppose that $\alpha\in\mathbb{R}\backslash\mathbb{Q},\,\beta\in\mathbb{R}$ and $0<\theta<10/1561$. Then there
  exist infinitely many primes $p$, which satisfy $p+2=\mathcal{P}_4$, such that
\begin{equation*}
   \big\|\alpha p^2+\beta\big\|<p^{-\theta}.
\end{equation*}
\end{theorem}

\noindent
\textbf{Remark.}
According to the work of Shi and Wu \cite{Shi-Wu-2013}, our improvement comes from using the methods developed
by Tolev \cite{Tolev-2000-2} with more delicate iterative techniques and various bounds for exponential sums,
combining with a version of the Lemma 2.2 of Bombieri and Iwaniec \cite{Bombieri-Iwaniec-1986}, while the
previous method in dealing exponential sum, e.g. \cite{Shi-Wu-2013}, is based on the traditional pattern of exponential sum estimates.

\section{Notation.}
Let $X$ be a sufficiently large real number. Set
\begin{equation}\label{parameter-1}
  \delta=0.307708,\quad \rho=0.23077,\quad \eta=0.076928,\quad \kappa=1.4999676,\quad 0<\theta<\frac{10}{1561}.
\end{equation}
Also, we put
\begin{align}\label{parameter-2}
  z=X^\eta,\quad y=X^{\rho},\quad D=X^\delta,\quad
   \Delta=\Delta(X)=X^{-\theta},\quad H=\Delta^{-1}\log^2X.
\end{align}
Throughout this paper, we always denote primes by $p$ and $q$. $\varepsilon$ always denotes an arbitrary small
positive constant, which may not be the same at different occurrences. As usual, we use
$\Omega(n),\varphi(n),\mu(n),\Lambda(n)$ to denote the number of prime factors of $n$ counted
according to multiplicity, Euler's function, M\"{o}bius' function, and Mangold's function, respectively. We
denote by $\tau_k(n)$ the number of solutions of the equation $m_1m_2\dots m_k=n$ in natural variables
$m_1,\dots,m_k$. Especially, we write $\tau_2(n)=\tau(n)$. Let $(m_1,m_2,\dots,m_k)$ and $[m_1,m_2,\dots,m_k]$
be the greatest common divisor and the least common multiple of $m_1,m_2,\dots,m_k$, respectively. Also,
we use $[x]$ and $\|x\|$, respectively, to denote the integer part of $x$ and the distance from $x$ to the nearest
integer. $f(x)\ll g(x)$ means that $f(x)=O(g(x))$; $f(x)\asymp g(x)$ means that $f(x)\ll g(x)\ll f(x)$;
$e(x)=e^{2\pi ix}$; $\mathscr{L}=\log X$. $\mathcal{P}_r$ always denotes an almost--prime with at most $r$ prime factors,
counted according to multiplicity.

\section{Preliminary Lemmas}
\begin{lemma}\label{Bombieri-Iwaniec-lemma}
Let $M\leqslant N<N_1\leqslant M_1$ and $a_n$ be any complex numbers. Then we have
\begin{equation*}
 \Bigg|\sum_{N<n\leqslant N_1}a_n\Bigg|\leqslant\int_{-\infty}^{+\infty}\mathcal{K}(\theta)
 \Bigg|\sum_{M<m\leqslant M_1}a_me(\theta m)\Bigg|\mathrm{d}\theta,
\end{equation*}
where
\begin{equation*}
 \mathcal{K}(\theta)=\min\Bigg(M_1-M+1,\,\frac{1}{\pi|\theta|},\,\frac{1}{\pi^2\theta^2}\Bigg),
\end{equation*}
which satisfies
\begin{equation}\label{Bom-Iwaniec-1}
 \int_{-\infty}^{+\infty}\mathcal{K}(\theta)\mathrm{d}\theta\leqslant3\log(2+M_1-M).
\end{equation}
\end{lemma}
\begin{proof}
 See Lemma 2.2 of Bombieri and Iwaniec \cite{Bombieri-Iwaniec-1986}.  $\hfill$
\end{proof}

\begin{lemma}\label{Heath-Brown-exponent-sum-fenjie}
Let $3\leqslant u<v<w<X$ and suppose that $w-\frac{1}{2}\in\mathbb{N}$, and that
$w\geqslant 4u^2,\,X\geqslant64 w^2u,\,v^3\geqslant32 X$. Assume further that $f(n)$ is a
complex--valued function. Then the sum
\begin{equation*}
 \sum_{\frac{X}{2}<n\leqslant X}\Lambda(n)f(n)
\end{equation*}
can be decomposed into $O(\log^{10}X)$ sums, each of which either of Type I:
\begin{equation*}
 \sum_{M<m\leqslant M_1}a_m\sum_{L<\ell\leqslant L_1}f(m\ell)
\end{equation*}
with $M<M_1\leqslant 2M,\,L<L_1\leqslant 2L,\,L\geqslant w$,  $a_m\ll m^\varepsilon,\,ML\asymp X$, or of Type II:
\begin{equation*}
 \sum_{M<m\leqslant M_1}a_m\sum_{L<\ell\leqslant L_1}b_\ell f(m\ell)
\end{equation*}
with $M<M_1\leqslant 2M,\,L<L_1\leqslant 2L,\,u\leqslant L\leqslant v$,
$a_m\ll m^\varepsilon,\,b_\ell \ll \ell^\varepsilon,\,ML\asymp X$.
\end{lemma}
\begin{proof}
 See Lemma 3 of Heath--Brown \cite{Heath-Brown-1983}.  $\hfill$
\end{proof}

\begin{lemma}\label{expo-Tri-lemma}
  For $P\geqslant1$, we have
  \begin{equation*}
     \sum_{1\leqslant n\leqslant P}e(\alpha n)\leqslant \min\bigg(P,\frac{1}{2\|\alpha\|}\bigg).
  \end{equation*}
\end{lemma}
\begin{proof}
See Lemma 4 of Chapter VI of Karatsuba \cite{Karatsuba-book}.  $\hfill$
\end{proof}

\begin{lemma}\label{Vaughan-Lemma-2.2}
  Suppose that $Y_1,Y_2,\alpha$ are real numbers with $Y_1\geqslant1, Y_2\geqslant1$, and that $|\alpha-a/q|\leqslant q^{-2}$ with $(a,q)=1$. Then we have
\begin{equation*}
  \sum_{n\leqslant Y_1}\min\bigg(\frac{Y_1Y_2}{n},\frac{1}{\|\alpha n\|}\bigg)
  \ll Y_1Y_2\bigg(\frac{1}{q}+\frac{1}{Y_2}+\frac{q}{Y_1Y_2}\bigg)\log(2Y_1q).
\end{equation*}
\end{lemma}
\begin{proof}
 See Lemma 2.2 of Vaughan \cite{Vaughan-book}.  $\hfill$
\end{proof}

\section{Proof of Theorem \ref{Theorem}}

As is shown in \cite{Todorova-Tolev-2010}, we take a periodic function $\chi(t)$ with period $1$ such that
\begin{equation}\label{chi-function}
  \begin{cases}
    0<\chi(t)<1, & \textrm{if}\,\, -\Delta<t<\Delta,  \\
    \quad \chi(t)=0, &  \textrm{if}\,\, \Delta\leqslant t\leqslant1-\Delta,
  \end{cases}
\end{equation}
which has a Fourier series
\begin{equation}\label{chi-Fourier-expan}
\chi(t)=\Delta+\sum_{|k|>0}g(k)e(kt)
\end{equation}
with coefficients satisfying
\begin{align}\label{Fourier-tail}
   & g(0)=\Delta, \nonumber \\
   & g(k)\ll \Delta, \quad \textrm{for all $k$}, \nonumber \\
   &  \sum_{|k|>H}|g(k)|\ll X^{-1}.
\end{align}
The existence of such a function is a consequence of a well known lemma of Vinogradov. For instance, one can see
Chapter I, \S2 in \cite{Karatsuba-book}. Consider the sum
\begin{equation}\label{Gamma-sum-def}
  \Gamma:=\Gamma(X)=\sum_{\substack{\frac{X}{2}<p\leqslant X\\ (p+2,P(z))=1}}\chi(\alpha p^2+\beta)
                    \mathcal{W}_p\log p,
\end{equation}
where
\begin{equation}\label{P(z)-def}
  P(z)=\prod_{2<p\leqslant z}p,
\end{equation}
and
\begin{equation}\label{W_p-wight}
  \mathcal{W}_p=1-\kappa\sum_{\substack{z<q\leqslant y\\ q|p+2}}\bigg(1-\frac{\log q}{\log y}\bigg).
\end{equation}
Let $\Gamma_1$ denote the sum of the terms of $\Gamma(X)$ in which $\mathcal{W}_p>0$. Then we have
\begin{equation*}
  \Gamma(X)\leqslant\Gamma_1.
\end{equation*}
If we denote by $\Gamma_2$ the sum of the terms of $\Gamma_1$ in which $\mu(p+2)=0$. It is easy to see that
\begin{align*}
             0 \leqslant \Gamma_2
  \ll &\,\, \sum_{q\geqslant z}\sum_{\substack{n\leqslant X\\ n+2\equiv0 \!\!\!\!\!\pmod {q^2} }}\log n
             \ll (\log X)\sum_{z\leqslant q\leqslant\sqrt{X+2}}\bigg(\frac{X}{q^2}+1\bigg)
                   \nonumber \\
  \ll &\,\, X^{1+\varepsilon}z^{-1}+X^{\frac{1}{2}+\varepsilon}\ll X^{1-\eta+\varepsilon}.
\end{align*}
By noting the fact that the contribution of the terms (if such terms exist) in $\Gamma_1$ for which
$X-2<p\leqslant X$ is $O(\log X)$, we deduce that
\begin{equation}\label{gamma<gamma_3+error}
  \Gamma\leqslant\Gamma_3+O(X^{1-\eta+\varepsilon}),
\end{equation}
where
\begin{equation*}
 \Gamma_3=\sum_{\substack{\frac{X}{2}<p\leqslant X-2\\ \mathcal{W}_p>0,\, \mu^2(p+2)=1\\ (p+2,P(z))=1}}
          \chi(\alpha p^2+\beta)\mathcal{W}_p\log p.
\end{equation*}
On one hand, if we assume that
\begin{equation*}
 \Gamma(X)\gg\frac{\Delta X}{\log X},
\end{equation*}
then from (\ref{gamma<gamma_3+error}), we get
\begin{equation*}
 \Gamma_3\gg\frac{\Delta X}{\log X},
\end{equation*}
and thus $\Gamma_3>0$. Hence there exists a prime $p$, which satisfies
\begin{equation}\label{condition-1}
  \frac{X}{2}<p\leqslant X-2, \quad \mathcal{W}_p>0,\quad \mu^2(p+2)=1,\quad (p+2,P(z))=1,
\end{equation}
and such that
\begin{equation}\label{exist-1}
  \chi(\alpha p^2+\beta)>0.
\end{equation}
Combining (\ref{chi-function}), (\ref{condition-1}) and (\ref{exist-1}), we can see that this prime $p$ satisfies
\begin{equation*}
  \big\|\alpha p^2+\beta \big\|<p^{-\theta}.
\end{equation*}
On the other hand, by the properties of the weights $\mathcal{W}_p$ (for example, one can see Chapter 9 of \cite{Halberstam-Richert}), it is easy to see that if $p$ satisfies (\ref{condition-1}), then
\begin{equation*}
  \Omega(p+2)=\sum_{\substack{q>z\\ q|p+2}}1<\frac{1}{\kappa}+\sum_{\substack{q>z\\ q|p+2}}\frac{\log q}{\log y}
             =\frac{1}{\kappa}+\frac{\log(p+2)}{\log y}\leqslant\frac{1}{\kappa}+\frac{1}{\rho}<5,
\end{equation*}
which implies $p+2=\mathcal{P}_4$. Therefore, in order to prove Theorem \ref{Theorem}, it is sufficient to show that there exists a sequence $\{X_j\}_{j=1}^\infty$, which satisfies
\begin{equation*}
  \lim_{j\to\infty}X_j=+\infty,\qquad \Gamma(X_j)\gg \frac{\Delta(X_j)X_j}{\log X_j},\quad j=1,2,3\dots.
\end{equation*}
By (\ref{Gamma-sum-def}) and (\ref{W_p-wight}), we can write $\Gamma$ as follows
\begin{equation}\label{Gamma=Psi-Phi}
   \Gamma=\Psi-\kappa\Phi,
\end{equation}
where
\begin{equation}\label{Psi-def}
  \Psi=\sum_{\substack{\frac{X}{2}<p\leqslant X\\(p+2,P(z))=1}}\chi(\alpha p^2+\beta)\log p,
\end{equation}
and
\begin{equation}\label{Phi-def}
  \Phi=\sum_{\substack{\frac{X}{2}<p\leqslant X\\(p+2,P(z))=1}}\chi(\alpha p^2+\beta)(\log p)
        \sum_{\substack{z<q\leqslant y\\ q|p+2}}\bigg(1-\frac{\log q}{\log y}\bigg).
\end{equation}
Next, we shall give lower bound estimate of $\Psi$ and upper bound estimate of $\Phi$ by using lower bound linear
sieve and upper bound linear sieve, respectively. First, we consider $\Psi$. Let $\lambda^-(d)$ be the lower bounds for Rosser's
weights of level $D$. Hence for any positive integer $d$, there holds
\begin{equation}\label{lower-condition}
  |\lambda^-(d)|\leqslant 1, \quad \lambda^-(d)=0 \quad \textrm{if}\quad d>D\,\,\,\,\textrm{or}\,\,\,\, \mu(d)=0,
\end{equation}
\begin{equation}\label{lambda_-upper}
  \sum_{d|n}\lambda^-(d)\leqslant\sum_{d|n}\mu(d)=
  \begin{cases}
    1, & \quad \textrm{if} \quad n=1,\\
    0, & \quad \textrm{if} \quad n\in\mathbb{N},\,n>1.\\
  \end{cases}
\end{equation}
Also, we shall use the fact if $2<s<4$, then there holds
\begin{equation}\label{lower-bound-sieve}
  \sum_{d|P(z)}\frac{\lambda^-(d)}{\varphi(d)}\geqslant\Pi(z)\bigg(\frac{2e^\gamma\log(s-1)}{s}
     +O\Big(\big(\log X\big)^{-1/3}\Big)\bigg),
\end{equation}
where
\begin{equation*}
  \Pi(z)=\prod_{2<p\leqslant z}\bigg(1-\frac{1}{p-1}\bigg).
\end{equation*}
Now, we take
\begin{equation}\label{s-def}
   s=\frac{\log D}{\log z}=\frac{\delta}{\eta}=\frac{76927}{19232}\in(2,4)
\end{equation}
in (\ref{lower-bound-sieve}). By (\ref{Psi-def}) and (\ref{lambda_-upper}), we obtain
\begin{align}\label{Psi>=Psi_1}
  \Psi = &\,\, \sum_{\frac{X}{2}<p\leqslant X}\chi(\alpha p^2+\beta)(\log p)\sum_{d|(p+2,P(z))}\mu(d)
                         \nonumber  \\
     \geqslant &\,\, \sum_{\frac{X}{2}<p\leqslant X}\chi(\alpha p^2+\beta)(\log p)\sum_{d|(p+2,P(z))}\lambda^-(d)
                         \nonumber  \\
     = &\,\, \sum_{d|P(z)}\lambda^-(d)\sum_{\substack{\frac{X}{2}<p\leqslant X\\ p+2\equiv0(\!\bmod d)}}
                             \chi(\alpha p^2+\beta)\log p
                          \nonumber  \\
     = &\,\, \Psi_1, \quad \textrm{say}.
\end{align}
From (\ref{chi-Fourier-expan}), we have
\begin{align*}
  \Psi_1 = &\,\, \sum_{d|P(z)}\lambda^-(d)\sum_{\substack{\frac{X}{2}<p\leqslant X\\ p+2\equiv0(\!\bmod d)}}
                 \Bigg(\Delta+\sum_{|k|>0}g(k)e(\alpha p^2k+\beta k)\Bigg)\log p
                           \nonumber  \\
  = &\,\, \Delta\sum_{d|P(z)}\lambda^-(d)\sum_{\substack{\frac{X}{2}<p\leqslant X\\ p+2\equiv0(\!\bmod d)}}\log p
                           \nonumber  \\
  &\,\,    +\sum_{d|P(z)}\lambda^-(d)\sum_{|k|>0}g(k)e(\beta k)
            \sum_{\substack{\frac{X}{2}<p\leqslant X\\ p+2\equiv0(\!\bmod d)}}e(\alpha p^2k)\log p
                            \nonumber  \\
  = &\,\, \Delta\sum_{d|P(z)}\lambda^-(d)\sum_{\substack{\frac{X}{2}<p\leqslant X\\ p+2\equiv0(\!\bmod d)}}\log p
                            \nonumber  \\
    &\,\,   +\Delta\sum_{d|P(z)}\lambda^-(d)\sum_{0<|k|\leqslant H}\Big(\Delta^{-1}g(k)e(\beta k)\Big)
            \sum_{\substack{\frac{X}{2}<p\leqslant X\\ p+2\equiv0(\!\bmod d)}}e(\alpha p^2k)\log p
                             \nonumber  \\
     &\,\,   +\sum_{d|P(z)}\lambda^-(d)\sum_{|k|>H}g(k)e(\beta k)
             \sum_{\substack{\frac{X}{2}<p\leqslant X\\ p+2\equiv0(\!\bmod d)}}e(\alpha p^2k)\log p.
\end{align*}
By (\ref{Fourier-tail}), and the fact that $\lambda^-(d)=0$ for $d>D$, we obtain
\begin{align*}
   &\,\,  \sum_{d|P(z)}\lambda^-(d)\sum_{|k|>H}g(k)e(\beta k)
          \sum_{\substack{\frac{X}{2}<p\leqslant X\\ p+2\equiv0(\!\bmod d)}}e(\alpha p^2k)\log p
                          \nonumber  \\
   \ll &\,\, \sum_{d|P(z)}\big|\lambda^-(d)\big| \sum_{|k|>H}\big|g(k)\big|
             \sum_{\substack{\frac{X}{2}<p\leqslant X\\ p+2\equiv0(\!\bmod d)}}\log p
    \ll \sum_{d\leqslant D}\frac{1}{\varphi(d)}\ll \log D\ll \log X.
\end{align*}
Therefore, we get
\begin{equation}\label{Psi_1-fenjie}
  \Psi_1=\Delta(\Psi_2+\Psi_3)+O(\log X),
\end{equation}
where
\begin{align} \label{Psi_3-def}
  &  \Psi_2=\sum_{d|P(z)}\lambda^-(d) \sum_{\substack{\frac{X}{2}<p\leqslant X\\ p+2\equiv0(\!\bmod d)}}\log p,
               \nonumber  \\
  & \Psi_3=\sum_{d|P(z)}\lambda^-(d)\sum_{0<|k|\leqslant H}c(k)
           \sum_{\substack{\frac{X}{2}<p\leqslant X\\ p+2\equiv0(\!\bmod d)}}e(\alpha p^2k)\log p ,
                 \\
  &  c(k)=\Delta^{-1}g(k)e(\beta k)\ll 1.  \nonumber
\end{align}
For $\Psi_2$, by Bombieri--Vinogradov's mean value theorem (See Chapter 28 of \cite{Davenport-book}) and (\ref{lower-condition}), we derive that
\begin{equation}\label{Psi_2-asymp}
  \Psi_2=\frac{X}{2}\sum_{d|P(z)}\frac{\lambda^-(d)}{\varphi(d)}+O\bigg(\frac{X}{\log^2X}\bigg).
\end{equation}
It follows from Mertens' prime number theorem (See \cite{Mertens}) that
\begin{equation}\label{Pi(z)-asymp}
  \Pi(z)\asymp \frac{1}{\log z}.
\end{equation}
Then from (\ref{lower-bound-sieve}), (\ref{Psi_2-asymp}) and (\ref{Pi(z)-asymp}), we obtain
\begin{equation}\label{Psi_2-lower}
  \Psi_2\geqslant e^\gamma X\Pi(z)\frac{\log(s-1)}{s}+O\bigg(\frac{X}{\log^{4/3}X}\bigg),
\end{equation}
where $s$ is defined by (\ref{s-def}). For $\Psi_3$, we shall investigate it in the next section.

Now, we study the sum $\Phi$, which is defined by (\ref{Phi-def}). We rewrite $\Phi$ in the following form
\begin{equation}\label{Phi-rewrite}
  \Phi=\sum_{z<q<y}\bigg(1-\frac{\log q}{\log y}\bigg)
       \sum_{\substack{\frac{X}{2}<p\leqslant X\\ p+2\equiv0(\!\bmod q)\\ (p+2,P(z))=1}}
        \chi(\alpha p^2+\beta)\log p.
\end{equation}
In order to give upper bound estimate of $\Phi$, we shall apply an upper bound linear sieve. Let $\lambda_q^+(d)$
be the upper bounds for Rosser's weights of level $D/q$. Hence for any positive integer $d$, we have
\begin{equation}\label{upper-condition}
  |\lambda_q^+(d)|\leqslant 1, \quad \lambda_q^+(d)=0 \quad \textrm{if}\quad d>\frac{D}{q}\,\,\,\,\textrm{or}\,\,\,\, \mu(d)=0,
\end{equation}
\begin{equation}\label{lambda^-upper}
  \sum_{d|n}\lambda_q^+(d)\geqslant\sum_{d|n}\mu(d)=
  \begin{cases}
    1, & \quad \textrm{if} \quad n=1,\\
    0, & \quad \textrm{if} \quad n\in\mathbb{N},\,n>1.\\
  \end{cases}
\end{equation}
Also, we shall use the fact, for $1<s_1<3$, there holds
\begin{equation}\label{upper-bound-sieve}
  \sum_{d|P(z)}\frac{\lambda_q^+(d)}{\varphi(d)}\leqslant\Pi(z)\bigg(\frac{2e^\gamma}{s_1}
     +O\Big(\big(\log X\big)^{-1/3}\Big)\bigg).
\end{equation}
For prime $q$ in the sum $\Phi$, we take
\begin{equation*}
  s_1=\frac{\log(D/q)}{\log z}.
\end{equation*}
Then it is easy to check that $1<s_1<3$, and thus (\ref{upper-bound-sieve}) holds. By (\ref{Phi-rewrite})--(\ref{lambda^-upper}), we obtain
\begin{align}\label{Phi<=Phi_1}
  \Phi = & \,\, \sum_{z<q<y}\bigg(1-\frac{\log q}{\log y}\bigg)
                \sum_{\substack{\frac{X}{2}<p\leqslant X\\ p+2\equiv0(\!\bmod q)}}\chi(\alpha p^2+\beta)(\log p)
                \sum_{d|(p+2,P(z))}\mu(d)
                         \nonumber  \\
 \leqslant & \,\, \sum_{z<q<y}\bigg(1-\frac{\log q}{\log y}\bigg)
                \sum_{\substack{\frac{X}{2}<p\leqslant X\\ p+2\equiv0(\!\bmod q)}}\chi(\alpha p^2+\beta)(\log p)
                \sum_{d|(p+2,P(z))}\lambda_q^+(d)
                           \nonumber  \\
   = & \,\, \sum_{z<q<y}\bigg(1-\frac{\log q}{\log y}\bigg)\sum_{d|P(z)}\lambda_q^+(d)
            \sum_{\substack{\frac{X}{2}<p\leqslant X\\ p+2\equiv0(\!\bmod {qd})}}\chi(\alpha p^2+\beta)(\log p)
                           \nonumber  \\
    = & \,\, \sum_{m\leqslant D}\nu(m)
             \sum_{\substack{\frac{X}{2}<p\leqslant X\\ p+2\equiv0(\!\bmod m)}}\chi(\alpha p^2+\beta)(\log p)
              \,\,=: \Phi_1,
\end{align}
where
\begin{equation}\label{nu(m)-def}
 \nu(m)=\sum_{\substack{z<q<y\\ d|P(z)\\m=dq}}\bigg(1-\frac{\log q}{\log y}\bigg)\lambda^+_q(d).
\end{equation}
If $m\leqslant z$, then $\nu(m)=0$. If $z<m\leqslant D$, by (\ref{P(z)-def}) and (\ref{nu(m)-def}) we know that the representation $m=dq$ with
$z<q<y$ and $d|P(z)$ is unique. Thus, it is easy to see that
\begin{equation}\label{nu(m)-upper}
 |\nu(m)|\leqslant1.
\end{equation}
From (\ref{chi-Fourier-expan}), we get
\begin{align*}
  \Phi_1 = & \,\, \sum_{m\leqslant D}\nu(m)
                  \sum_{\substack{\frac{X}{2}<p \leqslant X\\ p+2\equiv0(\!\bmod m)}}
                  \bigg(\Delta+\sum_{|k|>0}g(k)e(\alpha p^2k+\beta k)\bigg)\log p
                          \nonumber  \\
  = & \,\, \Delta\sum_{m\leqslant D}\nu(m)\sum_{\substack{\frac{X}{2}<p \leqslant X\\ p+2\equiv0(\!\bmod m)}}\log p
                           \nonumber  \\
     & \,\, + \sum_{m\leqslant D}\nu(m)\sum_{|k|>0}g(k)e(\beta k)
           \sum_{\substack{\frac{X}{2}<p \leqslant X\\ p+2\equiv0(\!\bmod m)}} e(\alpha p^2k)\log p
                           \nonumber  \\
  = & \,\, \Delta\sum_{m\leqslant D}\nu(m)\sum_{\substack{\frac{X}{2}<p \leqslant X\\ p+2\equiv0(\!\bmod m)}}\log p
                           \nonumber  \\
  & \,\, +\Delta \sum_{m\leqslant D}\nu(m)\sum_{0<|k|\leqslant H}\Big(\Delta^{-1}g(k)e(\beta k)\Big)
           \sum_{\substack{\frac{X}{2}<p \leqslant X\\ p+2\equiv0(\!\bmod m)}} e(\alpha p^2k)\log p
                            \nonumber  \\
  & \,\, +\sum_{m\leqslant D}\nu(m)\sum_{|k|>H}g(k)e(\beta k)
          \sum_{\substack{\frac{X}{2}<p \leqslant X\\ p+2\equiv0(\!\bmod m)}} e(\alpha p^2k)\log p.
\end{align*}
By (\ref{Fourier-tail}) and (\ref{nu(m)-upper}), we get
 \begin{align*}
   &\,\,  \sum_{m\leqslant D}\nu(m)\sum_{|k|>H}g(k)e(\beta k)
          \sum_{\substack{\frac{X}{2}<p\leqslant X\\ p+2\equiv0(\!\bmod m)}}e(\alpha p^2k)\log p
                          \nonumber  \\
   \ll &\,\, \sum_{|k|>H}\big|g(k)\big| \sum_{m\leqslant D}
             \sum_{\substack{\frac{X}{2}<p\leqslant X\\ p+2\equiv0(\!\bmod m)}}\log p
    \ll \sum_{m\leqslant D}\frac{1}{\varphi(m)}\ll \log X.
\end{align*}
Thus, we derive that
\begin{equation}\label{Phi_1-fenjie}
  \Phi_1=\Delta\big(\Phi_2+\Phi_3\big)+O(\log X),
\end{equation}
where
\begin{align}\label{Phi_3-def}
  &  \Phi_2 =\sum_{m\leqslant D}\nu(m)\sum_{\substack{\frac{X}{2}<p \leqslant X\\ p+2\equiv0(\!\bmod m)}}\log p,
                \nonumber  \\
  & \Phi_3=\sum_{m\leqslant D}\nu(m)\sum_{0<|k|\leqslant H}c(k)
      \sum_{\substack{\frac{X}{2}<p\leqslant X\\ p+2\equiv0(\!\bmod m)}}e(\alpha p^2k)\log p,
                   \\
  & c(k)=\Delta^{-1}g(k)e(\beta k)\ll1.    \nonumber
\end{align}
By Bombieri--Vinogradov's mean value theorem and (\ref{nu(m)-upper}), we have
\begin{equation}\label{Phi_2-asymp}
  \Phi_2=\frac{X}{2}\sum_{m\leqslant D}\frac{\nu(m)}{\varphi(m)}+O\bigg(\frac{X}{\log^2 X}\bigg).
\end{equation}
Using (\ref{upper-bound-sieve}) and (\ref{nu(m)-def}), we obtain
\begin{align}\label{Phi_2-innner}
            \sum_{m\leqslant D}\frac{\nu(m)}{\varphi(m)}
   = & \,\, \sum_{z<q<y}\bigg(1-\frac{\log q}{\log y}\bigg)\sum_{d|P(z)}\frac{\lambda_q^+(d)}{\varphi(qd)}
                 \nonumber  \\
   = & \,\, \sum_{z<q<y}\bigg(1-\frac{\log q}{\log y}\bigg)\frac{1}{q-1}
            \sum_{d|P(z)}\frac{\lambda_q^+(d)}{\varphi(d)}
                 \nonumber  \\
 \leqslant & \,\,\sum_{z<q<y}\bigg(1-\frac{\log q}{\log y}\bigg)\frac{1}{q-1}
      \Pi(z)\Bigg(2e^\gamma\bigg(\frac{\log(D/q)}{\log z}\bigg)^{-1}
                   +O\Big(\big(\log X\big)^{-1/3}\Big)\Bigg) .
\end{align}
Therefore, by (\ref{Pi(z)-asymp}), (\ref{Phi_2-asymp}) and (\ref{Phi_2-innner}), we have
\begin{equation}\label{Phi_2-upper-z}
  \Phi_2\leqslant e^\gamma X\Pi(z)\sum_{z<q<y}\bigg(1-\frac{\log q}{\log y}\bigg)\frac{1}{q-1}
                     \bigg(\frac{\log(D/q)}{\log z}\bigg)^{-1}
                      +O\Bigg(\frac{X}{(\log X)^{\frac{4}{3}-\varepsilon}}\Bigg).
\end{equation}

Now, we find a lower bound for the sum $\Gamma$. From (\ref{Gamma=Psi-Phi}), (\ref{Psi>=Psi_1}), (\ref{Psi_1-fenjie}), (\ref{Psi_2-lower}), (\ref{Phi<=Phi_1}), (\ref{Phi_1-fenjie}) and (\ref{Phi_2-upper-z}), we derive that
\begin{equation}\label{Gamma-lower-1}
  \Gamma\geqslant e^\gamma\Delta X\Pi(z)\mathfrak{S}
                 +O\Bigg(\frac{\Delta X}{(\log X)^{\frac{4}{3}-\varepsilon}}\Bigg)
                 +O\Big(\Delta\big|\Psi_3-\kappa\Phi_3\big|\Big),
\end{equation}
where
\begin{equation*}
  \mathfrak{S}=\frac{\log(s-1)}{s}-\kappa\sum_{z<q<y}\bigg(1-\frac{\log q}{\log y}\bigg)\frac{1}{q-1}
               \bigg(\frac{\log(D/q)}{\log z}\bigg)^{-1},\quad s=\frac{\log D}{\log z}.
\end{equation*}
Moreover, by partial summation and the prime number theorem, it is easy to show that
\begin{equation}\label{S-tihuan}
  \mathfrak{S}=\mathfrak{S}_0+O\bigg(\frac{1}{\log X}\bigg),
\end{equation}
where
\begin{equation*}
  \mathfrak{S}_0=\frac{\log(s-1)}{s}-\kappa\eta\int_{\eta}^{\rho}\bigg(\frac{1}{u}-\frac{1}{\rho}\bigg)
                  \frac{1}{\delta-u}\mathrm{d}u.
\end{equation*}
According to simple numerical calculation, we know that
\begin{equation*}
  \mathfrak{S}_0\geqslant 0.000032113949.
\end{equation*}
From (\ref{Pi(z)-asymp}), (\ref{Gamma-lower-1}) and (\ref{S-tihuan}), we obtain
\begin{equation}\label{Gamma-lower-2}
  \Gamma\geqslant e^\gamma\Delta X\Pi(z)\mathfrak{S}_0
                 +O\Bigg(\frac{\Delta X}{(\log X)^{\frac{4}{3}-\varepsilon}}\Bigg)
                 +O\Big(\Delta\big|\Psi_3-\kappa\Phi_3\big|\Big).
\end{equation}
We shall illustrate that if $X$ runs over a suitable sequence, which tends to infinity, then the second error term
in (\ref{Gamma-lower-2}) can be absorbed. Hence we need the following lemma.

\begin{lemma}\label{exponential-lemma}
  Suppose that $\alpha\in\mathbb{R}\backslash\mathbb{Q}$ and $\delta,\theta, D,H$ are defined in
  (\ref{parameter-1}) and (\ref{parameter-2}). Let $\xi(d), c(k)$ be complex numbers defined
  for $d\leqslant D,\,0<|k|\leqslant H$, respectively, which satisfy
  \begin{equation*}
     \xi(d)\ll1,\qquad c(k)\ll 1.
  \end{equation*}
  Then there exists a sequence $\{X_j\}_{j=1}^\infty$ satisfying $\lim\limits_{j\to \infty}X_j=+\infty$, such that
  the sum $S(X)$ defined by
   \begin{equation}\label{exponential-type}
     S(X)=\sum_{d\leqslant D}\xi(d)\sum_{1\leqslant|k|\leqslant H}c(k)
     \sum_{\substack{\frac{X}{2}<p\leqslant X\\ p+2\equiv0(\!\bmod d)}}(\log p)e(\alpha p^2k)
  \end{equation}
  satisfies
   \begin{equation*}
     S(X_j)\ll \frac{X_j}{\log^2X_j}, \qquad j=1,2,3,\dots.
  \end{equation*}
\end{lemma}
The proof of Lemma \ref{exponential-lemma} will be given in the next section. From (\ref{Psi_3-def}) and (\ref{Phi_3-def}), we know that $\Psi_3-\kappa\Phi_3$ can be represented as a sum of
type (\ref{exponential-type}) with
\begin{equation*}
  \xi(d)=\lambda^*(d)-\kappa\nu(d),
\end{equation*}
where
\begin{equation*}
  \lambda^*(d)=
  \begin{cases}
     \lambda^-(d), & \textrm{if}\,\, d|P(z), \\
     \quad0, & \textrm{otherwise}.
  \end{cases}
\end{equation*}
According to Lemma \ref{exponential-lemma} and (\ref{Gamma-lower-2}), there exists a
sequence $\{X_j\}_{j=1}^\infty$, which tends to infinity, such that
\begin{equation}\label{Gamma-lower-3}
  \Gamma(X_j)\geqslant e^\gamma\Delta X_j\Pi(z)\mathfrak{S}_0
                 +O\Bigg(\frac{\Delta X_j}{(\log X_j)^{\frac{4}{3}-\varepsilon}}\Bigg).
\end{equation}
From (\ref{Pi(z)-asymp}) and (\ref{Gamma-lower-3}), we know that there exists a positive constant $c>0$ such that
  \begin{equation*}
     \Gamma(X_j)\geqslant \frac{c\Delta(X_j) X_j}{\log X_j}>0, \qquad j=1,2,3,\dots.
  \end{equation*}
This completes the proof of Theorem \ref{Theorem}.

\section{Proof of Lemma \ref{exponential-lemma} }
In this section, we shall prove Lemma \ref{exponential-lemma}. Since $\alpha\in\mathbb{R}\backslash\mathbb{Q}$, by
Dirichlet's approximation theorem, there exist infinitely many integers $A$ and natural numbers $Q$ with $(A,Q)=1$ such that
\begin{equation*}
   \bigg|\alpha-\frac{A}{Q}\bigg|<\frac{1}{Q^2}.
\end{equation*}
For each such $Q$, we choose $X$ in a suitable way, i.e. as in (\ref{Q-choose}). In this way, we construct
our sequence $\{X_j\}_{j=1}^\infty$.

First, we have
\begin{equation*}
   S(X)=W+O\big(HX^{\frac{1}{2}+\varepsilon}\big),
\end{equation*}
where
\begin{equation*}
   W=\sum_{\frac{X}{2}<n\leqslant X}\Lambda(n)\sum_{1\leqslant|k|\leqslant H}c(k)e(\alpha n^2k)
       \sum_{\substack{d\leqslant D\\ d|n+2\\ 2\nmid d}}\xi(d).
\end{equation*}
According to Lemma \ref{Heath-Brown-exponent-sum-fenjie}, by taking $u=2^{-7}X^{\frac{\delta}{2}},\,v=2^7X^\frac{1}{3},\,w=X^{\frac{1}{2}-\frac{\delta}{4}}$, it is easy to see that
the sum $W$ can be decompose into $O(\log^{10}X)$ sums, each of which either of Type I:
\begin{equation*}
   S_I=\sum_{M<m\leqslant M_1}a_m\sum_{\substack{L<\ell\leqslant L_1\\ \frac{X}{2}<m\ell \leqslant X}}
        \sum_{1\leqslant|k|\leqslant H}c(k)e(\alpha m^2\ell^2k)
        \sum_{\substack{d\leqslant D\\ d|m\ell+2\\ 2\nmid d}}\xi(d)
\end{equation*}
with $M_1\leqslant2M,\,L_1\leqslant2L,\,L\geqslant w,\,a_m\ll m^{\varepsilon},\,ML\asymp X$; or of Type II:
\begin{equation*}
   S_{II}=\sum_{M<m\leqslant M_1}a_m\sum_{\substack{L<\ell\leqslant L_1\\ \frac{X}{2}<m\ell \leqslant X}}b_\ell
        \sum_{1\leqslant|k|\leqslant H}c(k)e(\alpha m^2\ell^2k)
        \sum_{\substack{d\leqslant D\\ d|m\ell+2\\ 2\nmid d}}\xi(d)
\end{equation*}
with $M_1\leqslant2M,\,L_1\leqslant2L,\,u\leqslant L\leqslant v,\,a_m\ll m^{\varepsilon},\,b_\ell\ll\ell^\varepsilon,\, ML\asymp X$.

Next, we shall deal with the sums of Type I and Type II in the following subsections, respectively.

\subsection{The Estimate of  Type II Sums}
In this subsection, we shall deal with the estimate of the sums of Type II. First, we have
\begin{align*}
  S_{II} = & \,\, \sum_{1\leqslant|k|\leqslant H}c(k)\sum_{M<m\leqslant M_1}a_m
           \sum_{\substack{L<\ell\leqslant L_1\\ \frac{X}{2}<m\ell\leqslant X}}b_\ell e(\alpha m^2\ell^2k)
           \sum_{\substack{d\leqslant D\\ d|m\ell+2\\ 2\nmid d}}\xi(d)
                        \nonumber \\
  \ll & \,\, X^\varepsilon \sum_{1\leqslant|k|\leqslant H}\sum_{M<m\leqslant M_1}
       \Bigg|\sum_{\substack{L<\ell\leqslant L_1\\ \frac{X}{2}<m\ell\leqslant X}}b_\ell e(\alpha m^2\ell^2k)
           \sum_{\substack{d\leqslant D\\ d|m\ell+2\\ 2\nmid d}}\xi(d) \Bigg|.
\end{align*}
By Cauchy's inequality, we get
\begin{align*}
  & \,\,|S_{II}|^2 \ll X^\varepsilon HM\sum_{1\leqslant|k|\leqslant H}\sum_{M<m\leqslant M_1}
                       \Bigg|\sum_{\substack{L<\ell\leqslant L_1\\ \frac{X}{2}<m\ell\leqslant X}}b_\ell
                       e(\alpha m^2\ell^2k)\sum_{\substack{d\leqslant D\\ d|m\ell+2\\ 2\nmid d}}\xi(d) \Bigg|^2
                                 \nonumber \\
 = & \,\, X^\varepsilon HM\sum_{1\leqslant|k|\leqslant H}\sum_{M<m\leqslant M_1}
            \sum_{\substack{L<\ell\leqslant L_1\\ \frac{X}{2}<m\ell\leqslant X}}
            \sum_{\substack{d_1,d_2\leqslant D\\ m\ell_1+2\equiv0(\!\bmod d_1)\\ m\ell_2+2\equiv0(\!\bmod d_2)\\
            (d_1d_1,2)=1}}b_{\ell_1}\overline{b_{\ell_2}}\xi(d_1)\overline{\xi(d_2)}
            e\big(\alpha m^2k(\ell_1^2-\ell_2^2)\big)
                                 \nonumber \\
  \ll & \, X^\varepsilon HM\sum_{1\leqslant|k|\leqslant H}
           \sum_{L<\ell_1,\ell_2\leqslant L_1}\sum_{\substack{d_1,d_2\leqslant D\\ (d_1d_2,2)=1\\ (\ell_1,d_1)=(\ell_2,d_2)=1}}
           \Bigg|\sum_{\substack{M'<m\leqslant M_1'\\ m\ell_1+2\equiv0(\!\bmod d_1)\\ m\ell_2+2\equiv0(\!\bmod d_2)}} e\big(\alpha m^2k(\ell_1^2-\ell_2^2)\big) \Bigg|
                                   \nonumber \\
  \ll & \, X^\varepsilon HM\sum_{1\leqslant|k|\leqslant H}  \sum_{L<\ell_1,\ell_2\leqslant L_1}
        \sum_{\substack{d_1,d_2\leqslant D\\ (d_1d_2,2)=1\\ (\ell_1,d_1)=(\ell_2,d_2)=1}}|\mathcal{V}|,
\end{align*}
where
\begin{equation*}
   \mathcal{V}=\sum_{\substack{M'<m\leqslant M_1'\\ m\ell_1+2\equiv0(\!\bmod d_1)\\ m\ell_2+2\equiv0(\!\bmod d_2)}} e\big(\alpha m^2k(\ell_1^2-\ell_2^2)\big),
\end{equation*}
\begin{equation*}
   M'=\max\bigg(M,\frac{X}{2\ell_1},\frac{X}{2\ell_2}\bigg),\qquad M_1'=\min\bigg(M_1,\frac{X}{\ell_1},\frac{X}{\ell_2}\bigg).
\end{equation*}
If the system of the congruence
\begin{equation}\label{congruence-sys-1}
  \begin{cases}
    m\ell_1+2\equiv0(\bmod d_1) \\
    m\ell_2+2\equiv0(\bmod d_2)
  \end{cases}
\end{equation}
has no solution, then $\mathcal{V}=0$. Assume that (\ref{congruence-sys-1}) has a solution.
Then there exists an $f_0=f_0(\ell_1,\ell_2,d_1,d_2)$ such that (\ref{congruence-sys-1}) is
equivalent to $m\equiv f_0(\bmod {[d_1,d_2]})$. In this case, we have
\begin{align*}
 |\mathcal{V}|= &\,\, \Bigg|\sum_{\substack{M'<m\leqslant M_1'\\ m\equiv f_0(\bmod {[d_1,d_2]})}}
                         e\big(\alpha m^2k(\ell_1^2-\ell_2^2)\big)\Bigg|
                                                     \nonumber \\
= &\,\,\Bigg|\sum_{\frac{M'-f_0}{[d_1,d_2]}<r\leqslant\frac{M_1'-f_0}{[d_1,d_2]}}
 e\Big(\alpha\big(f_0+r[d_1,d_2]\big)^2k\big(\ell_1^2-\ell_2^2\big)\Big)\Bigg|
                     \nonumber \\
 = &\,\,\Bigg|\sum_{\frac{M'-f_0}{[d_1,d_2]}<r\leqslant\frac{M_1'-f_0}{[d_1,d_2]}}
    e\Big(\alpha\big(r^2[d_1,d_2]^2+2f_0r[d_1,d_2]\big)k\big(\ell_1^2-\ell_2^2\big)\Big)\Bigg|
                     \nonumber \\
 = &\,\,\Bigg|\sum_{R<r\leqslant R_1}
    e\Big(\alpha\big(r^2[d_1,d_2]^2+2f_0r[d_1,d_2]\big)k\big(\ell_1^2-\ell_2^2\big)\Big)\Bigg|,
\end{align*}
where
\begin{equation*}
   R=\frac{M'-f_0}{[d_1,d_2]},\qquad R_1=\frac{M_1'-f_0}{[d_1,d_2]}.
\end{equation*}
The contribution of $\mathcal{V}$ with $\ell_1=\ell_2$ to $|S_{II}|^2$ is
\begin{align*}
  \ll &\,\, X^\varepsilon HM^2\sum_{1\leqslant k\leqslant H}
            \sum_{\substack{L<\ell_1,\ell_2\leqslant L_1\\ \ell_1=\ell_2}}
            \sum_{d_1,d_2\leqslant D}\frac{1}{[d_1,d_2]}
                       \nonumber \\
   \ll &\,\,X^\varepsilon H^2M^2L\sum_{d_1,d_2\leqslant D}\frac{1}{[d_1,d_2]}
             \ll X^\varepsilon H^2M^2L\sum_{h\leqslant D^2}\frac{\tau^2(h)}{h}
                     \nonumber \\
   \ll &\,\, X^\varepsilon H^2M^2L\ll X^{1+\varepsilon}H^2M.
\end{align*}
Therefore, we have
\begin{equation*}
   |S_{II}|^2\ll X^{1+\varepsilon}H^2M+X^\varepsilon HM\sum_{1\leqslant k\leqslant H}
                \sum_{\substack{d_1,d_2\leqslant D\\ (d_1d_2,2)=1}}
                \sum_{\substack{L<\ell_1,\ell_2\leqslant L_1\\ (\ell_1,d_1)=(\ell_2,d_2)=1\\ \ell_1\not=\ell_2}}
                |\mathcal{V}|.
\end{equation*}
Moreover, by Cauchy's inequality again, we obtain
\begin{align}\label{S_(II)-fourth}
  |S_{II}|^4 \ll &\,\, X^{2+\varepsilon}H^4M^2+X^{\varepsilon}H^3M^2\sum_{1\leqslant k\leqslant H}
                     \Bigg(\sum_{\substack{d_1,d_2\leqslant D\\ (d_1d_2,2)=1}}
   \sum_{\substack{L<\ell_1,\ell_2\leqslant L_1\\ (\ell_1,d_1)=(\ell_2,d_2)=1\\ \ell_1\not=\ell_2}}
                     |\mathcal{V}|\Bigg)^2
                        \nonumber \\
     \ll &\,\, X^{2+\varepsilon}H^4M^2+X^{\varepsilon}H^3M^2\sum_{1\leqslant k\leqslant H}
     \Bigg(\sum_{\substack{d_1,d_2\leqslant D\\ (d_1d_2,2)=1}}\frac{1}{[d_1,d_2]} \Bigg)
                          \nonumber \\
      &\,\, \quad \times \sum_{\substack{d_1,d_2\leqslant D\\ (d_1d_2,2)=1}} [d_1,d_2]
            \Bigg(\sum_{\substack{L<\ell_1,\ell_2\leqslant L_1\\ (\ell_1,d_1)=(\ell_2,d_2)=1\\ \ell_1\not=\ell_2}}|\mathcal{V}|\Bigg)^2
                         \nonumber \\
     \ll &\,\,  X^{2+\varepsilon}H^4M^2+X^{\varepsilon}H^3M^2L^2
                \Bigg(\sum_{\substack{d_1,d_2\leqslant D\\ (d_1d_2,2)=1}}\frac{1}{[d_1,d_2]} \Bigg)
                      \nonumber \\
      &\,\, \quad \times \sum_{1\leqslant k\leqslant H} \sum_{\substack{d_1,d_2\leqslant D\\ (d_1d_2,2)=1}}
                  [d_1,d_2]\sum_{\substack{L<\ell_1,\ell_2\leqslant L_1\\ (\ell_1,d_1)=(\ell_2,d_2)=1\\ \ell_1\not=\ell_2}}|\mathcal{V}|^2
                        \nonumber \\
    \ll &\,\, X^{2+\varepsilon}H^4M^2+X^{\varepsilon}H^3M^2L^2
             \sum_{1\leqslant k\leqslant H} \sum_{\substack{d_1,d_2\leqslant D\\ (d_1d_2,2)=1}}
                  [d_1,d_2]\sum_{\substack{L<\ell_1,\ell_2\leqslant L_1\\ (\ell_1,d_1)=(\ell_2,d_2)=1\\ \ell_1\not=\ell_2}}|\mathcal{V}|^2
                          \nonumber \\
    \ll &\,\,  X^{2+\varepsilon}H^4M^2+X^{\varepsilon}H^3M^2L^2\cdot\Sigma_0,
\end{align}
where
\begin{align*}
   \Sigma_0 = & \,\, \sum_{1\leqslant k\leqslant H} \sum_{\substack{d_1,d_2\leqslant D\\ (d_1d_2,2)=1}}[d_1,d_2]
                 \sum_{\substack{L<\ell_1,\ell_2\leqslant L_1\\ (\ell_1,d_1)=(\ell_2,d_2)=1\\ \ell_1\not=\ell_2}}
                          \nonumber \\
   & \,\, \times \sum_{R<r_1,r_2\leqslant R_1}e\Big(\alpha\big((r_1^2-r_2^2)[d_1,d_2]^2+2f_0(r_1-r_2)[d_1,d_2]\big)k\big(\ell_1^2-\ell_2^2\big)\Big).
\end{align*}
For $\Sigma_0$, we have
\begin{align*}
   \Sigma_0 = & \,\, \sum_{1\leqslant k\leqslant H} \sum_{\substack{d_1,d_2\leqslant D\\ (d_1d_2,2)=1}}[d_1,d_2]
                 \sum_{\substack{L<\ell_1,\ell_2\leqslant L_1\\ (\ell_1,d_1)=(\ell_2,d_2)=1\\ \ell_1\not=\ell_2}}
                          \nonumber \\
   & \,\, \times \sum_{s_1,s_2}\Bigg(\sum_{\substack{R<r_1,r_2\leqslant R_1\\ r_1-r_2=s_1\\ r_1+r_2=s_2}}1\Bigg)
   e\Big(\alpha\big(s_1s_2[d_1,d_2]^2+2f_0s_1[d_1,d_2]\big)k\big(\ell_1^2-\ell_2^2\big)\Big)
                             \nonumber \\
   = & \,\, \sum_{1\leqslant k\leqslant H} \sum_{\substack{d_1,d_2\leqslant D\\ (d_1d_2,2)=1}}[d_1,d_2]
                 \sum_{\substack{L<\ell_1,\ell_2\leqslant L_1\\ (\ell_1,d_1)=(\ell_2,d_2)=1\\ \ell_1\not=\ell_2}}
                              \nonumber \\
   & \,\, \times  \sum_{\substack{s_1,s_2:\,\,s_1\equiv s_2(\!\bmod 2)\\ 2R<s_2+s_1\leqslant 2R_1\\ 2R<s_2-s_1\leqslant 2R_1}}
                 e\Big(\alpha\big(s_1s_2[d_1,d_2]^2+2f_0s_1[d_1,d_2]\big)k\big(\ell_1^2-\ell_2^2\big)\Big)
                              \nonumber \\
  = & \,\,  \sum_{1\leqslant k\leqslant H}\sum_{\substack{d_1,d_2\leqslant D\\ (d_1d_2,2)=1}}[d_1,d_2]
             \sum_{\substack{L<\ell_1,\ell_2\leqslant L_1\\ (\ell_1,d_1)=(\ell_2,d_2)=1\\ \ell_1\not=\ell_2}}
             \sum_{|s_1|\leqslant 2R_1-2R}e\Big(2\alpha f_0s_1[d_1,d_2]k(\ell_1^2-\ell_2^2)\Big)
                                \nonumber \\
  & \,\, \times   \sum_{\substack{s_2:\,\,s_2\equiv s_1(\!\bmod 2)\\
                         2R<s_2+s_1\leqslant 2R_1\\ 2R<s_2-s_1\leqslant 2R_1}}
                  e\Big(\alpha s_1s_2[d_1,d_2]^2k\big(\ell_1^2-\ell_2^2\big)\Big) .
\end{align*}
Set
\begin{equation*}
   D_0= X^{\frac{50}{3}\theta}.
\end{equation*}
Then we divide $\Sigma_0$ into two parts
\begin{equation}\label{Sigma_0-two-parts}
   \Sigma_0=\Sigma_1+\Sigma_2,
\end{equation}
where $\Sigma_1$ denotes the part of $\Sigma_0$ which satisfies $[d_1,d_2]\leqslant D_0$, while $\Sigma_2$ denotes
the remaining part of $\Sigma_0$ which satisfies $[d_1,d_2]>D_0$. We set $s_2=s_1+2t$ in $\Sigma_1$ and $\Sigma_2$,
and derive that
\begin{align}\label{Sigma_1-upper-1}
                  \Sigma_1
 \leqslant& \,\, \sum_{1\leqslant k\leqslant H}\sum_{\substack{d_1,d_2\leqslant D_0\\ (d_1d_2,2)=1}}[d_1,d_2]
     \sum_{\substack{L<\ell_1,\ell_2\leqslant L_1\\ (\ell_1,d_1)=(\ell_2,d_2)=1\\ \ell_1\not=\ell_2}}
                       \nonumber \\
  &\,\, \quad\times \sum_{|s_1|\leqslant 2R_1-2R}\Bigg|\sum_{R'<t\leqslant R_1'}
      e\Big(2\alpha s_1t[d_1,d_2]^2k(\ell_1^2-\ell_2^2)\Big)\Bigg|
\end{align}
and
\begin{align}\label{Sigma_2-upper-1}
                  \Sigma_2
 \leqslant& \,\, \sum_{1\leqslant k\leqslant H}\sum_{\substack{d_1,d_2\leqslant D\\ (d_1d_2,2)=1\\ [d_1,d_2]>D_0}}[d_1,d_2]
     \sum_{\substack{L<\ell_1,\ell_2\leqslant L_1\\ (\ell_1,d_1)=(\ell_2,d_2)=1\\ \ell_1\not=\ell_2}}
                       \nonumber \\
  &\,\, \quad\times \sum_{|s_1|\leqslant 2R_1-2R}\Bigg|\sum_{R'<t\leqslant R_1'}
      e\Big(2\alpha s_1t[d_1,d_2]^2k(\ell_1^2-\ell_2^2)\Big)\Bigg|,
\end{align}
where
\begin{equation*}
   R'=\max\big(R-s_1,R\big),\qquad R_1'=\min\big(R_1-s_1,R_1\big).
\end{equation*}
First, we consider the upper bound for $\Sigma_1$. Let $\Sigma_1^{(1)}$ and $\Sigma_1^{(2)}$ denote the contribution of the right--hand side of (\ref{Sigma_1-upper-1}) for $s_1\not=0$ and $s_1=0$, respectively. Trivially, there holds
\begin{equation}\label{Sigma_1(2)-upper}
   \Sigma_1^{(2)}\ll HML^2\sum_{\substack{d_1,d_2\leqslant D_0\\(d_1d_2,2)=1}}1\ll
   HML^2D_0^2\ll D_0^2HXL.
\end{equation}
For $\Sigma_1^{(1)}$, by Lemma \ref{expo-Tri-lemma} we have
\begin{align}\label{Sigma_1(1)<tau_7}
 \Sigma_1^{(1)} \ll &\, \sum_{1\leqslant k\leqslant H}\sum_{d_1,d_2\leqslant D_0}[d_1,d_2]
                        \sum_{\substack{L<\ell_1,\ell_2\leqslant L_1\\ \ell_1\not=\ell_2}}
                            \nonumber \\
    &\,\qquad\times  \sum_{0<|s|\leqslant\frac{2M}{[d_1,d_2]}}
           \min\Bigg(\frac{M}{[d_1,d_2]},\frac{1}{\big\|2\alpha s[d_1,d_2]^2k(\ell_1^2-\ell_2^2)\big\|}\Bigg)
                              \nonumber \\
   \ll &\, \sum_{1\leqslant k\leqslant H}\sum_{h\leqslant D_0^2}h
           \Bigg(\sum_{\substack{d_1,d_2\leqslant D_0\\ [d_1,d_2]=h}}1\Bigg)
           \sum_{\substack{L<\ell_1,\ell_2\leqslant L_1\\ \ell_1\not=\ell_2}}
           \sum_{0<|s|\leqslant\frac{2M}{h}}
           \min\Bigg(\frac{M}{h},\frac{1}{\big\|2\alpha sh^2k(\ell_1^2-\ell_2^2)\big\|}\Bigg)
                               \nonumber \\
   \ll &\, D_0^2 \sum_{1\leqslant k\leqslant H}\sum_{h\leqslant D_0^2}
           \sum_{\substack{L<\ell_1,\ell_2\leqslant L_1\\ \ell_1\not=\ell_2}}
           \sum_{0<|s|\leqslant2M}\min\Bigg(M,\frac{1}{\big\|2\alpha sh^2k(\ell_1^2-\ell_2^2)\big\|}\Bigg)
                              \nonumber \\
   \ll &\, D_0^2\sum_{1\leqslant k\leqslant H}\sum_{h\leqslant D_0^2}
            \sum_{t_1,t_2}\Bigg(\sum_{\substack{L<\ell_1,\ell_2\leqslant L_1\\ \ell_1-\ell_2=t_1\\ \ell_1+\ell_2=t_2\\ \ell_1\not=\ell_2}}1\Bigg)
            \sum_{0<|s|\leqslant2M}\min\Bigg(M,\frac{1}{\big\|2\alpha sh^2kt_1t_2\big\|}\Bigg)
                              \nonumber \\
   \ll &\, D_0^2\sum_{1\leqslant k\leqslant H}\sum_{h\leqslant D_0^2}
           \sum_{\substack{1\leqslant|t_1|\leqslant L\\ 1\leqslant t_2\leqslant4L}}
            \sum_{0<|s|\leqslant2M}\min\Bigg(M,\frac{1}{\big\|2\alpha sh^2kt_1t_2\big\|}\Bigg)
                               \nonumber \\
   \ll &\, D_0^2\sum_{1\leqslant k\leqslant H}\sum_{h\leqslant D_0^2} \sum_{1\leqslant t_1,t_2\leqslant4L}
            \sum_{1\leqslant|s|\leqslant2M}\min\Bigg(M,\frac{1}{\big\|2\alpha sh^2kt_1t_2\big\|}\Bigg)
                              \nonumber \\
    \ll &\, D_0^2\sum_{1\leqslant m\leqslant64D_0^4HML^2}\tau_7(m)
            \min\bigg(M,\frac{1}{\|\alpha m\|}\bigg).
\end{align}
By Lemma \ref{Vaughan-Lemma-2.2}, we have
\begin{align}\label{tau_7-upper}
      & \,\, \sum_{1\leqslant m\leqslant64D_0^4HML^2}\tau_7(m)\min\bigg(M,\frac{1}{\|\alpha m\|}\bigg)
                            \nonumber \\
  \ll & \, X^\varepsilon\sum_{1\leqslant m\leqslant64D_0^4HML^2}
          \min\Bigg(\frac{64D_0^4HM^2L^2}{m},\frac{1}{\|\alpha m\|}\Bigg)
                            \nonumber \\
  \ll &\, X^\varepsilon D_0^4HM^2L^2
          \bigg(\frac{1}{Q}+\frac{1}{M}+\frac{Q}{D_0^4HM^2L^2}\bigg)
                            \nonumber \\
  \ll &\, X^\varepsilon \bigg(\frac{HX^2D_0^4}{Q}+HXLD_0^4+Q\bigg).
\end{align}
Combining (\ref{Sigma_1-upper-1}), (\ref{Sigma_1(2)-upper}), (\ref{Sigma_1(1)<tau_7}) and (\ref{tau_7-upper}), and
by noting the fact that $ML\asymp X$, we get
\begin{align}\label{Sigma_1-upper-2}
   \Sigma_1 \ll X^\varepsilon\Big(HX^2D_0^6Q^{-1}+HXLD_0^6+QD_0^2\Big).
\end{align}

Now, we consider the estimate of $\Sigma_2$. According to (\ref{Sigma_2-upper-1}), by a splitting argument, we have
\begin{equation}\label{Sigma_2-splitting}
   \Sigma_2\ll \mathscr{L} \max_{D_0\ll T\ll D^2}\big(T\Sigma_2^{(1)}\big),
\end{equation}
where
\begin{align*}
   \Sigma_2^{(1)}=\Sigma_2^{(1)}(T)=& \,\,\sum_{1\leqslant k \leqslant H}
   \sum_{\substack{d_1,d_2\leqslant D\\ (d_1d_2,2)=1\\ T<[d_1,d_2]\leqslant2T}}
   \sum_{\substack{L<\ell_1,\ell_2\leqslant L_1\\ (\ell_1,d_1)=(\ell_2,d_2)=1\\ \ell_1\not=\ell_2}}
                 \nonumber \\
   & \,\,\quad\times\sum_{|s_1|\leqslant2R_1-2R}\Bigg|\sum_{R'<t\leqslant R_1'}
           e\Big(2\alpha s_1t[d_1,d_2]^2k\big(\ell_1^2-\ell_2^2\big)\Big)\Bigg|.
\end{align*}
By Lemma \ref{Bombieri-Iwaniec-lemma}, we have
\begin{align}\label{Sigma_2-upper-12}
   \Sigma_2^{(1)}=\Sigma_2^{(1)}(T)\leqslant & \,\,\sum_{1\leqslant k \leqslant H}
         \sum_{\substack{d_1,d_2\leqslant D\\ T<[d_1,d_2]\leqslant2T}}
         \sum_{\substack{L<\ell_1,\ell_2\leqslant L_1\\ \ell_1\not=\ell_2}}\sum_{|s|\leqslant\frac{2M}{T}}
                                 \nonumber \\
       & \,\,\qquad \times\int_{-\infty}^{+\infty}\mathcal{K}(\theta)
              \Bigg|\sum_{\frac{M}{4T}<t\leqslant\frac{4M}{T}}
              e\Big(2\alpha st[d_1,d_2]^2k\big(\ell_1^2-\ell_2^2\big)+\theta t\Big)\Bigg|\mathrm{d}\theta
                          \nonumber \\
       =: &\,\, \int_{-\infty}^{+\infty}\mathcal{K}(\theta)\cdot\Sigma_{2}^{(2)}(\theta,T) \mathrm{d}\theta,
\end{align}
where
\begin{equation*}
   \mathcal{K}(\theta)=\min\bigg(\frac{15M}{4T}+1,\frac{1}{\pi|\theta|},\frac{1}{\pi^2\theta^2}\bigg),
\end{equation*}
and
\begin{align*}
        \Sigma_{2}^{(2)}(\theta,T)
= &\,\, \sum_{1\leqslant k \leqslant H} \sum_{\substack{d_1,d_2\leqslant D\\ T<[d_1,d_2]\leqslant2T}}
         \sum_{\substack{L<\ell_1,\ell_2\leqslant L_1\\ \ell_1\not=\ell_2}}\sum_{|s|\leqslant\frac{2M}{T}}
                  \nonumber \\
  &\,\,\qquad\times   \Bigg|\sum_{\frac{M}{4T}<t\leqslant\frac{4M}{T}}
              e\Big(2\alpha st[d_1,d_2]^2k\big(\ell_1^2-\ell_2^2\big)+\theta t\Big)\Bigg|.
\end{align*}
According to (\ref{Bom-Iwaniec-1}) and (\ref{Sigma_2-upper-12}), it is easy to see that
\begin{equation}\label{Sigma_2(1)-max}
   \Sigma_2^{(1)}\ll \mathscr{L}\max_{0\leqslant\theta\leqslant1}\Sigma_2^{(2)}(\theta,T).
\end{equation}
For $\Sigma_2^{(2)}(\theta,T)$, we have
\begin{align}\label{Sigma_2(2)-S_3}
           \Sigma_2^{(2)}(\theta,T)
 = & \,\, \sum_{1\leqslant k\leqslant H}\sum_{T<h\leqslant2T}
          \Bigg(\sum_{\substack{d_1,d_2\leqslant D\\ [d_1,d_2]=h}}1\Bigg)
          \sum_{\substack{L<\ell_1,\ell_2\leqslant L_1\\ \ell_1\not=\ell_2}}\sum_{|s|\leqslant\frac{2M}{T}}
                     \nonumber \\
 & \,\,\qquad \times\Bigg|\sum_{\frac{M}{4T}<t\leqslant\frac{4M}{T}}
              e\Big(2\alpha sth^2k\big(\ell_1^2-\ell_2^2\big)+\theta t\Big)\Bigg|
                      \nonumber \\
  \ll & \,\,  \sum_{1\leqslant k\leqslant H}\sum_{T<h\leqslant2T} \tau^2(h)
             \sum_{\substack{L<\ell_1,\ell_2\leqslant L_1\\ \ell_1\not=\ell_2}}\sum_{|s|\leqslant\frac{2M}{T}}
             \Bigg|\sum_{\frac{M}{4T}<t\leqslant\frac{4M}{T}}
              e\Big(2\alpha sth^2k\big(\ell_1^2-\ell_2^2\big)+\theta t\Big)\Bigg|
                        \nonumber \\
   = & \,\, \sum_{1\leqslant k\leqslant H}\sum_{T<h\leqslant2T} \tau^2(h)\sum_{t_1,t_2}
             \Bigg(\sum_{\substack{L<\ell_1,\ell_2\leqslant L_1\\ \ell_1-\ell_2=t_1\\ \ell_1+\ell_2=t_2\\ \ell_1\not=\ell_2}}1\Bigg)
                        \nonumber \\
    & \,\,\qquad \times  \sum_{|s|\leqslant\frac{2M}{T}}
             \Bigg|\sum_{\frac{M}{4T}<t\leqslant\frac{4M}{T}}
              e\Big(2\alpha sth^2kt_1t_2+\theta t\Big)\Bigg|
                      \nonumber \\
    \ll & \,\, \sum_{1\leqslant k\leqslant H}\sum_{T<h\leqslant2T} \tau^2(h)
               \sum_{1\leqslant|t_1|,|t_2|\leqslant4L}\sum_{|s|\leqslant\frac{2M}{T}}
               \Bigg|\sum_{\frac{M}{4T}<t\leqslant\frac{4M}{T}}e\Big(2\alpha sth^2kt_1t_2+\theta t\Big)\Bigg|
                      \nonumber \\
    \ll & \,\, \mathscr{L}^3 HML^2+\sum_{1\leqslant k\leqslant H}\sum_{T<h\leqslant2T} \tau^2(h)
                 \sum_{1\leqslant|t_1|,|t_2|\leqslant4L}\sum_{1\leqslant|s|\leqslant\frac{2M}{T}}
                      \nonumber \\
          & \,\,\qquad\qquad\qquad\qquad \times \Bigg|\sum_{\frac{M}{4T}<t\leqslant\frac{4M}{T}}
                         e\Big(2\alpha sth^2kt_1t_2+\theta t\Big)\Bigg|
                          \nonumber \\
    \ll & \,\, \mathscr{L}^3 HML^2+\sum_{1\leqslant k\leqslant H}\sum_{T<h\leqslant2T} \tau^2(h)
               \sum_{1\leqslant|m|\leqslant\frac{32ML^2}{T}}\tau_3(|m|)
                         \nonumber \\
    & \,\,\qquad\qquad\qquad\qquad \times\Bigg|\sum_{\frac{M}{4T}<t\leqslant\frac{4M}{T}}
                         e\Big(2\alpha th^2km+\theta t\Big)\Bigg|
                           \nonumber \\
    = & \,\, \mathscr{L}^3 HML^2+\Sigma_2^{(3)},\quad \textrm{say}.
\end{align}
It follows from Cauchy's inequality that
\begin{align}\label{Sigma_2(3)-suqare-S_4}
              \big(\Sigma_2^{(3)}\big)^2
   \ll & \,\, H\sum_{1\leqslant k\leqslant H}
             \Bigg(\sum_{T<h\leqslant2T}\tau^2(h)\sum_{1\leqslant|m|\leqslant\frac{32ML^2}{T}}\tau_3(|m|)
              \Bigg|\sum_{\frac{M}{4T}<t\leqslant\frac{4M}{T}}e\Big(2\alpha th^2km+\theta t\Big)\Bigg| \Bigg)^2
                    \nonumber \\
  \ll & \,\, H\sum_{1\leqslant k\leqslant H}\Bigg(\sum_{T<h\leqslant2T}\tau^4(h) \Bigg)
              \sum_{T<h\leqslant2T}\Bigg(\sum_{1\leqslant|m|\leqslant\frac{32ML^2}{T}}\tau_3(|m|)
                     \nonumber \\
  & \,\,\qquad\qquad\times\Bigg|\sum_{\frac{M}{4T}<t\leqslant\frac{4M}{T}}
              e\Big(2\alpha th^2km+\theta t\Big)\Bigg| \Bigg)^2
                       \nonumber \\
  \ll & \,\, H\Bigg(\sum_{T<h\leqslant2T}\tau^4(h) \Bigg)
             \Bigg(\sum_{1\leqslant|m|\leqslant\frac{32ML^2}{T}}\tau_3^2(|m|)\Bigg)
             \sum_{1\leqslant k\leqslant H}\sum_{T<h\leqslant2T}
                        \nonumber \\
   & \,\, \qquad\qquad\times\sum_{1\leqslant|m|\leqslant\frac{32ML^2}{T}}
          \Bigg|\sum_{\frac{M}{4T}<t\leqslant\frac{4M}{T}}e\Big(2\alpha th^2km+\theta t\Big)\Bigg|^2
                       \nonumber \\
  \ll & \,\, \mathscr{L}^{23} HML^2 \sum_{1\leqslant k\leqslant H}\sum_{T<h\leqslant2T}
             \sum_{1\leqslant|m|\leqslant\frac{32ML^2}{T}}
              \Bigg|\sum_{\frac{M}{4T}<t\leqslant\frac{4M}{T}}e\Big(2\alpha th^2km+\theta t\Big)\Bigg|^2
                       \nonumber \\
  = & \,\,\mathscr{L}^{23} HML^2\cdot \Sigma_2^{(4)},\quad \textrm{say}.
\end{align}
For $\Sigma_2^{(4)}$, we have
\begin{align}\label{Sigma_2(4)-S_5}
              \Sigma_2^{(4)}
    = & \,\, \sum_{1\leqslant k\leqslant H}\sum_{T<h\leqslant2T}\sum_{1\leqslant|m|\leqslant\frac{32ML^2}{T}}
             \sum_{\frac{M}{4T}<t_1,t_2\leqslant\frac{4M}{T}}
              e\Big(\big(2\alpha h^2km+\theta\big)(t_1-t_2)\Big)
                       \nonumber \\
    \ll  & \,\, \sum_{1\leqslant k\leqslant H}\sum_{1\leqslant|m|\leqslant\frac{32ML^2}{T}}
                \sum_{\frac{M}{4T}<t_1,t_2\leqslant\frac{4M}{T}}
                \Bigg|\sum_{T<h\leqslant2T} e\Big(\big(2\alpha h^2km\big)(t_1-t_2)\Big)\Bigg|
                        \nonumber \\
   \ll  & \,\, \frac{HM^2L^2}{T} +\frac{M}{T}\sum_{1\leqslant k\leqslant H}
              \sum_{1\leqslant|m|\leqslant\frac{32ML^2}{T}}\sum_{1\leqslant|n|\leqslant\frac{4M}{T}}
               \Bigg|\sum_{T<h\leqslant2T} e\big(2\alpha h^2kmn\big)\Bigg|
                          \nonumber \\
  \ll  & \,\, \frac{HM^2L^2}{T} +\frac{M}{T}\sum_{1\leqslant k\leqslant H}
              \sum_{1\leqslant|s|\leqslant\frac{256M^2L^2}{T^2}}\tau_3(|s|)
              \Bigg|\sum_{T<h\leqslant2T} e\big(\alpha h^2ks\big)\Bigg|
                          \nonumber \\
  = & \,\, \frac{HM^2L^2}{T}+\frac{M}{T}\cdot\Sigma_2^{(5)},\quad \textrm{say}.
\end{align}
By Cauchy's inequality, we deduce that
\begin{align}\label{Sigma_2(5)-suqare-S_6}
             (\Sigma_2^{(5)})^2
 \ll & \,\, H\sum_{1\leqslant k\leqslant H}\Bigg(\sum_{1\leqslant s\leqslant\frac{256M^2L^2}{T^2}}
            \tau_3(s)\Bigg|\sum_{T<h\leqslant2T}e(\alpha h^2ks)\Bigg|\Bigg)^2
                  \nonumber \\
 \ll & \,\, H\Bigg(\sum_{1\leqslant s\leqslant\frac{256M^2L^2}{T^2}}\tau_3^2(s)\Bigg)
            \sum_{1\leqslant k\leqslant H}\sum_{1\leqslant s\leqslant\frac{256M^2L^2}{T^2}}
            \Bigg|\sum_{T<h\leqslant2T}e(\alpha h^2ks)\Bigg|^2
                   \nonumber \\
 \ll & \,\, \frac{\mathscr{L}^{8} HM^2L^2}{T^2}\sum_{1\leqslant k\leqslant H}
            \sum_{1\leqslant s\leqslant\frac{256M^2L^2}{T^2}}\sum_{T<h_1,h_2\leqslant2T}
             e\big(\alpha ks\big(h_1^2-h_2^2\big)\big)
                   \nonumber \\
  \ll & \,\, \frac{\mathscr{L}^{8} H^2M^4L^4}{T^3}+\frac{\mathscr{L}^{8} HM^2L^2}{T^2}
             \sum_{1\leqslant k\leqslant H}\sum_{1\leqslant s\leqslant\frac{256M^2L^2}{T^2}}
             \sum_{\substack{T<h_1,h_2\leqslant2T\\ h_1\not=h_2}}
             e\big(\alpha ks\big(h_1^2-h_2^2\big)\big)
                    \nonumber \\
  = & \,\, \frac{\mathscr{L}^{8} H^2M^4L^4}{T^3}+
           \frac{\mathscr{L}^{8} HM^2L^2}{T^2}\cdot \Sigma_2^{(6)}, \quad \textrm{say}.
\end{align}
For $\Sigma_2^{(6)}$, from Lemma \ref{expo-Tri-lemma} we have
\begin{align}\label{Sigma_2(6)-upper-1}
              \Sigma_2^{(6)}
  = & \,\, \sum_{1\leqslant k\leqslant H}\sum_{t_1,t_2}
           \Bigg(\sum_{\substack{T<h_1,h_2\leqslant2T\\ h_1-h_2=t_1\\ h_1+h_2=t_2\\ h_1\not=h_2}}1\Bigg)
           \sum_{1\leqslant s\leqslant\frac{256M^2L^2}{T^2}}e(\alpha kst_1t_2)
                  \nonumber \\
\ll & \,\, \sum_{1\leqslant k\leqslant H}\sum_{1\leqslant t_1,t_2\leqslant4T}
             \Bigg|\sum_{1\leqslant s\leqslant\frac{256M^2L^2}{T^2}}e(\alpha kst_1t_2)\Bigg|
                   \nonumber \\
\ll & \,\, \sum_{1\leqslant k\leqslant H}\sum_{1\leqslant t_1,t_2\leqslant4T}
             \min\bigg(\frac{M^2L^2}{T^2},\frac{1}{\|\alpha kt_1t_2\|}\bigg)
                   \nonumber \\
\ll & \,\,\sum_{1\leqslant n\leqslant16HT^2}\tau_3(n) \min\bigg(\frac{M^2L^2}{T^2},\frac{1}{\|\alpha n \|}\bigg).
\end{align}
It follows from Lemma \ref{Vaughan-Lemma-2.2} that
\begin{align}\label{Sigma_2(6)-upper-2}
    & \,\, \sum_{1\leqslant n\leqslant16HT^2}\tau_3(n) \min\bigg(\frac{M^2L^2}{T^2},\frac{1}{\|\alpha n \|}\bigg)
                    \nonumber \\
  \ll & \,\, X^\varepsilon \sum_{1\leqslant n\leqslant16HT^2}
             \min\bigg(\frac{16HM^2L^2}{n},\frac{1}{\|\alpha n \|}\bigg)
                     \nonumber \\
  \ll & \,\, X^\varepsilon HM^2L^2\bigg(\frac{1}{Q}+\frac{T^2}{M^2L^2}+\frac{Q}{HM^2L^2}\bigg)
                     \nonumber \\
  \ll & \,\,X^\varepsilon \big(HX^2Q^{-1}+HT^2+Q\big).
\end{align}
From (\ref{Sigma_2(5)-suqare-S_6}), (\ref{Sigma_2(6)-upper-1}) and (\ref{Sigma_2(6)-upper-2}), we obtain
\begin{equation}\label{Sigma_2(5)-hui}
  \Sigma_2^{(5)}\ll X^\varepsilon\Bigg(\frac{HX^2}{T^{3/2}}+\frac{HX^2}{TQ^{1/2}}+HX
  +\frac{H^{\frac{1}{2}}XQ^{\frac{1}{2}}}{T}\Bigg).
\end{equation}
Putting (\ref{Sigma_2(5)-hui}) into (\ref{Sigma_2(4)-S_5}), we get
\begin{equation}\label{Sigma_2(4)-hui}
  \Sigma_2^{(4)}\ll X^\varepsilon\Bigg(\frac{HX^2}{T}+\frac{HX^2M}{T^{5/2}}+\frac{HX^2M}{T^2Q^{1/2}}
  +\frac{HXM}{T}+\frac{H^{\frac{1}{2}}XQ^{\frac{1}{2}}M}{T^2}\Bigg).
\end{equation}
Combining (\ref{Sigma_2(3)-suqare-S_4}) and (\ref{Sigma_2(4)-hui}), one has
\begin{equation}\label{Sigma_2(3)-hui}
  \Sigma_2^{(3)}\ll X^\varepsilon\Bigg(\frac{HX^\frac{3}{2}L^{\frac{1}{2}}}{T^{1/2}}+\frac{HX^2}{T^{5/4}}
                     +\frac{H^{\frac{1}{2}}X^2}{TQ^{1/4}}
                      +\frac{H^{\frac{3}{4}}X^{\frac{3}{2}}Q^{\frac{1}{4}}}{T}\Bigg).
\end{equation}
Inserting (\ref{Sigma_2(3)-hui}) into (\ref{Sigma_2(2)-S_3}), we derive that
\begin{equation*}
  \Sigma_2^{(2)}(\theta,T)\ll X^\varepsilon\Bigg(HXL+\frac{HX^{\frac{3}{2}}L^{\frac{1}{2}}}{T^{1/2}}
                    +\frac{HX^2}{T^{5/4}}+\frac{H^{\frac{1}{2}}X^2}{TQ^{1/4}}
                    +\frac{H^{\frac{3}{4}}X^{\frac{3}{2}}Q^{\frac{1}{4}}}{T}\Bigg),
\end{equation*}
which combines (\ref{Sigma_2-splitting}) and (\ref{Sigma_2(1)-max}) to get
\begin{align}\label{Sigma_2-typeII-last}
  \Sigma_2  \ll & \,\, X^\varepsilon\max_{D_0\ll T\ll D^2}
                  \Big(HXLT+HX^{\frac{3}{2}}L^{\frac{1}{2}}T^{\frac{1}{2}}
                  +HX^2T^{-\frac{1}{4}}+H^{\frac{1}{2}}X^2Q^{-\frac{1}{4}}
                  +H^{\frac{3}{4}}X^{\frac{3}{2}}Q^{\frac{1}{4}}\Big)
                        \nonumber \\
    \ll & \,\, X^\varepsilon\Big(HXLD^2+HX^{\frac{3}{2}}L^{\frac{1}{2}}D+HX^2D_0^{-\frac{1}{4}}
               +H^{\frac{1}{2}}X^2Q^{-\frac{1}{4}}+H^{\frac{3}{4}}X^{\frac{3}{2}}Q^{\frac{1}{4}}\Big).
\end{align}
From (\ref{Sigma_0-two-parts}), (\ref{Sigma_1-upper-2}) and (\ref{Sigma_2-typeII-last}), we obtain
\begin{align*}
              \Sigma_0
  \ll & \,\, X^{\varepsilon}\Big(D_0^6HX^2Q^{-1}+D_0^6HXL+D_0^2Q+HXLD^2+HX^{\frac{3}{2}}L^{\frac{1}{2}}D
                   \nonumber \\
   & \,\,\qquad +HX^2D_0^{-\frac{1}{4}}+H^{\frac{1}{2}}X^2Q^{-\frac{1}{4}}
                +H^{\frac{3}{4}}X^{\frac{3}{2}}Q^{\frac{1}{4}}\Big),
\end{align*}
which combines (\ref{S_(II)-fourth}) yields
\begin{align}\label{S_II-upper-last}
              S_{II}
  \ll & \,\,  X^\varepsilon \Big(HXL^{-\frac{1}{2}}+D_0^{\frac{3}{2}}HXQ^{-\frac{1}{4}}
                                 +D_0^{\frac{3}{2}}HX^{\frac{3}{4}}L^{\frac{1}{4}}
                                 +D_0^{\frac{1}{2}}Q^{\frac{1}{4}}H^{\frac{3}{4}}X^{\frac{1}{2}}
                                 +HX^{\frac{3}{4}}L^{\frac{1}{4}}D^{\frac{1}{2}}
                        \nonumber \\
     & \,\,\qquad  +HX^{\frac{7}{8}}L^{\frac{1}{8}}D^{\frac{1}{4}}+HXD_0^{-\frac{1}{16}}
                   +H^{\frac{7}{8}}XQ^{-\frac{1}{16}}+H^{\frac{15}{16}}X^{\frac{7}{8}}Q^{\frac{1}{16}}\Big)
                          \nonumber \\
  \ll & \,\,  X^\varepsilon \Big(HXu^{-\frac{1}{2}}+D_0^{\frac{3}{2}}HXQ^{-\frac{1}{4}}
                                 +D_0^{\frac{3}{2}}HX^{\frac{3}{4}}v^{\frac{1}{4}}
                                 +D_0^{\frac{1}{2}}Q^{\frac{1}{4}}H^{\frac{3}{4}}X^{\frac{1}{2}}
                                 +HX^{\frac{3}{4}}v^{\frac{1}{4}}D^{\frac{1}{2}}
                        \nonumber \\
     & \,\,\qquad  +HX^{\frac{7}{8}}v^{\frac{1}{8}}D^{\frac{1}{4}}+HXD_0^{-\frac{1}{16}}
                   +H^{\frac{7}{8}}XQ^{-\frac{1}{16}}+H^{\frac{15}{16}}X^{\frac{7}{8}}Q^{\frac{1}{16}}\Big).
\end{align}

\subsection{The Estimate of  Type I Sums }
In this subsection, we shall deal with the estimate of the sums of Type I. First, we have
\begin{equation*}
  S_{I}=\sum_{1\leqslant|k|\leqslant H}c(k)\sum_{\substack{d\leqslant D\\ (d,2)=1}}\xi(d)
        \sum_{M<m\leqslant M_1}a_m\sum_{\substack{L'<\ell\leqslant L_1'\\ m\ell+2\equiv0(\!\bmod d)}}
        e\big(\alpha m^2\ell^2k\big),
\end{equation*}
where
\begin{equation*}
 L'=\max\bigg(L,\frac{X}{2m}\bigg),\qquad L_1'=\min\bigg(L_1,\frac{X}{m}\bigg).
\end{equation*}
By a splitting argument, there holds
\begin{equation}\label{S_I-splitting}
 S_I\ll X^\varepsilon\cdot\max_{1\ll T\ll D}\Sigma_3,
\end{equation}
where
\begin{equation*}
\Sigma_3=\sum_{1\leqslant k\leqslant H}\sum_{\substack{T<d\leqslant2T\\ (d,2)=1}}
\sum_{\substack{M<m\leqslant M_1\\ (m,d)=1}}\Bigg|\sum_{\substack{L'<\ell\leqslant L_1'\\
      m\ell+2\equiv 0(\!\bmod d)}}e\big(\alpha m^2\ell^2k\big)\Bigg|.
\end{equation*}
For $(m,d)=1$, there exists $\overline{m}$, which satisfies $0\leqslant\overline{m}\leqslant d-1$, such that
$m\overline{m}\equiv1(\bmod d)$. Therefore, the equation $m\ell+2\equiv 0(\bmod d)$ is equivalent to
$\ell\equiv-2\overline{m}(\bmod d)$, i.e. $\ell=-2\overline{m}+dr$ for some $r\in\mathbb{Z}$. Then it follows from Cauchy's inequality that
\begin{align*}
              \big(\Sigma_3\big)^2
   \ll & \,\, HMT\sum_{1\leqslant k\leqslant H}\sum_{\substack{T<d\leqslant2T\\ (d,2)=1}}
              \sum_{\substack{M<m\leqslant M_1\\ (m,d)=1}}\Bigg|\sum_{\substack{L'<\ell\leqslant L_1'\\ m\ell+2\equiv0(\!\bmod d)}}e\big(\alpha m^2\ell^2k\big)\Bigg|^2
                      \nonumber \\
   \ll & \,\, HMT\sum_{1\leqslant k\leqslant H}\sum_{\substack{T<d\leqslant2T\\ (d,2)=1}}
              \sum_{\substack{M<m\leqslant M_1\\ (m,d)=1}}\Bigg|\sum_{\frac{L'+2\overline{m}}{d}<r\leqslant\frac{L_1'+2\overline{m}}{d}}
              e\big(\alpha m^2(-2\overline{m}+dr)^2k\big)\Bigg|^2
                      \nonumber \\
    =  & \,\, HMT\sum_{1\leqslant k\leqslant H}\sum_{\substack{T<d\leqslant2T\\ (d,2)=1}}
              \sum_{\substack{M<m\leqslant M_1\\ (m,d)=1}}
              \Bigg|\sum_{\frac{L'+2\overline{m}}{d}<r\leqslant\frac{L_1'+2\overline{m}}{d}}
              e\big(\alpha m^2(d^2r^2-4\overline{m}dr)k\big)\Bigg|^2
                      \nonumber \\
    =  & \,\, HMT\sum_{1\leqslant k\leqslant H}\sum_{\substack{T<d\leqslant2T\\ (d,2)=1}}
              \sum_{\substack{M<m\leqslant M_1\\ (m,d)=1}}
                      \nonumber \\
    & \,\,\qquad \quad\quad\times\sum_{\frac{L'+2\overline{m}}{d}<r_1,r_2\leqslant\frac{L_1'+2\overline{m}}{d}}
              e\Big(\alpha m^2\big(d^2(r_1^2-r_2^2)-4\overline{m}d(r_1-r_2)\big)k\Big).
\end{align*}
Set
\begin{equation*}
 R=\frac{L'+2\overline{m}}{d},\qquad \quad R_1=\frac{L_1'+2\overline{m}}{d}.
\end{equation*}
Then we have
\begin{align}\label{Sigma_3-square-(1)}
              \big(\Sigma_3\big)^2
  \ll & \,\,  X^\varepsilon H^2M^2LT+HMT\Bigg|\sum_{1\leqslant k\leqslant H}
              \sum_{\substack{T<d\leqslant2T\\ (d,2)=1}}\sum_{\substack{M<m\leqslant M_1\\ (m,d)=1}}
                       \nonumber \\
     & \,\,\qquad  \times  \sum_{\substack{R<r_1,r_2\leqslant R_1\\ r_1\not=r_2}}
               e\Big(\alpha m^2\big(d^2(r_1^2-r_2^2)-4\overline{m}d(r_1-r_2)\big)k\Big)\Bigg|
                       \nonumber \\
  \ll & \,\, X^\varepsilon H^2M^2LT+HMT\cdot\big|\Sigma_3^{(1)}\big|,
\end{align}
where
\begin{equation*}
\Sigma_3^{(1)}=\sum_{1\leqslant k\leqslant H}
              \sum_{\substack{T<d\leqslant2T\\ (d,2)=1}}\sum_{\substack{M<m\leqslant M_1\\ (m,d)=1}}
              \sum_{\substack{R<r_1,r_2\leqslant R_1\\ r_1\not=r_2}}
              e\Big(\alpha m^2\big(d^2(r_1^2-r_2^2)-4\overline{m}d(r_1-r_2)\big)k\Big).
\end{equation*}
For $\Sigma_3^{(1)}$, we have
\begin{align}\label{Sigma_3(1)-upper}
        \Sigma_3^{(1)}
  = & \,\, \sum_{1\leqslant k\leqslant H}\sum_{\substack{T<d\leqslant2T\\ (d,2)=1}}
           \sum_{\substack{M<m\leqslant M_1\\ (m,d)=1}}\sum_{s_1,s_2}
           \Bigg(\sum_{\substack{R<r_1,r_2\leqslant R_1\\ r_1-r_2=s_1\\ r_1+r_2=s_2\\ r_1\not=r_2}}1\Bigg)
           e\Big(\alpha m^2\big(d^2s_1s_2-4\overline{m}ds_1\big)k\Big)
                           \nonumber \\
  = & \,\, \sum_{1\leqslant k\leqslant H}\sum_{\substack{T<d\leqslant2T\\ (d,2)=1}}
           \sum_{\substack{M<m\leqslant M_1\\ (m,d)=1}}
           \sum_{\substack{s_1,\,\,s_2\\ 2R<s_1+s_2\leqslant2R_1\\ 2R<s_2-s_1\leqslant2R_1\\
              s_1\equiv s_2(\!\bmod 2),\, s_1\not=0}} e\Big(\alpha m^2\big(d^2s_1s_2-4\overline{m}ds_1\big)k\Big)
                          \nonumber \\
  \ll & \,\, \sum_{1\leqslant k\leqslant H}\sum_{\substack{T<d\leqslant2T\\ (d,2)=1}}
           \sum_{\substack{M<m\leqslant M_1\\ (m,d)=1}}\sum_{1\leqslant|s_1|\leqslant\frac{4L}{T}}
           \Bigg|\sum_{\substack{s_2:\,\,s_2\equiv s_1(\!\bmod 2)\\ 2R-s_1<s_2\leqslant2R_1-s_1\\
           2R+s_1<s_2\leqslant2R_1+s_1}}e\big(\alpha m^2d^2s_1s_2k\big)\Bigg|
                           \nonumber \\
   \ll & \,\, \sum_{1\leqslant k\leqslant H}\sum_{\substack{T<d\leqslant2T\\ (d,2)=1}}
              \sum_{\substack{M<m\leqslant M_1\\ (m,d)=1}}\sum_{1\leqslant|s_1|\leqslant\frac{4L}{T}}
              \Bigg|\sum_{\substack{R-s_1<t\leqslant R_1-s_1\\ R<t\leqslant R_1}}
              e\big(2\alpha m^2d^2s_1tk\big)\Bigg|.
\end{align}
Next, we will discuss the estimate of the right--hand side of (\ref{Sigma_3(1)-upper}) in two cases.

\textbf{Case 1.} Suppose that $MT\leqslant D_0$, and under this condition, there holds $1\ll M,T\ll D_0$.
By Lemma \ref{Vaughan-Lemma-2.2}, we have
\begin{align}\label{Sigma_3(1)-case-1}
             \Sigma_3^{(1)}
  \ll & \,\,  \sum_{1\leqslant k\leqslant H}\sum_{T<d\leqslant2T}\sum_{M<m\leqslant M_1}
              \sum_{1\leqslant s\leqslant8L}\min\bigg(L,\frac{1}{\|2\alpha m^2d^2sk \|}\bigg)
                        \nonumber \\
 \ll & \,\,  \sum_{1\leqslant k\leqslant H}\sum_{T<d\leqslant2T}\sum_{M<m\leqslant M_1}
              \sum_{1\leqslant s\leqslant8L}
              \min\bigg(\frac{256HM^2T^2L^2}{2 m^2d^2sk},\frac{1}{\|2\alpha m^2d^2sk\|}\bigg)
                       \nonumber \\
 \ll & \,\, \sum_{1\leqslant n\leqslant256HM^2T^2L}\tau_7(n)
            \min\bigg(\frac{HM^2T^2L^2}{n},\frac{1}{\|\alpha n\|}\bigg)
                       \nonumber \\
 \ll & \,\, X^\varepsilon HX^2T^2\bigg(\frac{1}{Q}+\frac{1}{L}+\frac{Q}{HX^2T^2}\bigg)
                       \nonumber \\
 \ll & \,\, X^\varepsilon\bigg(\frac{HX^2D_0^2}{Q}+HX(MT)T+Q\bigg)
                        \nonumber \\
 \ll & \,\, X^\varepsilon\bigg(\frac{HX^2D_0^2}{Q}+HXD_0^2+Q\bigg).
\end{align}
From (\ref{S_I-splitting}), (\ref{Sigma_3-square-(1)}) and (\ref{Sigma_3(1)-case-1}), we derive that,
under the condition $MT\leqslant D_0$, there holds
\begin{equation}\label{S_I-upper-case-1}
S_I\ll X^\varepsilon\Big(HX^{\frac{1}{2}}D_0^{\frac{1}{2}}+HXD_0^{\frac{3}{2}}Q^{-\frac{1}{2}}
            +HX^{\frac{1}{2}}D_0^{\frac{3}{2}}+H^{\frac{1}{2}}D_0^{\frac{1}{2}}Q^{\frac{1}{2}}\Big).
\end{equation}

\textbf{Case 2.} Now, we suppose that $MT>D_0$. Set
\begin{equation*}
  R'=\max\big(R,R-s_1\big),\qquad R_1'=\min\big(R_1,R_1-s_1\big).
\end{equation*}
Applying Lemma \ref{Bombieri-Iwaniec-lemma} to (\ref{Sigma_3(1)-upper}), we have
\begin{align}\label{Sigma_3(1)-Bom-Iwanoec}
              \Sigma_3^{(1)}
  \ll & \,\,  \sum_{1\leqslant k\leqslant H}\sum_{T<d\leqslant2T}\sum_{M<m\leqslant M_1}
              \sum_{1\leqslant|s_1|\leqslant\frac{4L}{T}}\Bigg|\sum_{R'<t\leqslant R_1'}
              e\big(2\alpha m^2d^2s_1tk\big)\Bigg|
                       \nonumber \\
  \ll & \,\, \sum_{1\leqslant k\leqslant H}\sum_{T<d\leqslant2T}\sum_{M<m\leqslant M_1}
              \sum_{1\leqslant|s_1|\leqslant\frac{4L}{T}}\int_{-\infty}^{+\infty}\mathcal{K}_1(\theta)
              \Bigg|\sum_{\frac{L}{4T}<t\leqslant\frac{4L}{T}}e\big(2\alpha m^2d^2s_1tk+\theta t\big)
              \Bigg|\mathrm{d}\theta
                       \nonumber \\
  = & \,\, \int_{-\infty}^{+\infty}\mathcal{K}_1(\theta)\cdot\Sigma_3^{(2)}(\theta,T)\mathrm{d}\theta,
\end{align}
where
\begin{equation*}
  \mathcal{K}_1(\theta)=\min\bigg(\frac{15L}{4T}+1,\frac{1}{\pi|\theta|},\frac{1}{\pi^2\theta^2}\bigg),
\end{equation*}
and
\begin{equation*}
      \Sigma_3^{(2)}(\theta,T)
   =  \sum_{1\leqslant k\leqslant H}\sum_{T<d\leqslant2T}\sum_{M<m\leqslant M_1}
      \sum_{1\leqslant|s_1|\leqslant\frac{4L}{T}}
      \Bigg|\sum_{\frac{L}{4T}<t\leqslant\frac{4L}{T}}e\big(2\alpha m^2d^2s_1tk+\theta t\big)\Bigg|.
\end{equation*}
According to (\ref{Bom-Iwaniec-1}) and (\ref{Sigma_3(1)-Bom-Iwanoec}), it is easy to see that
\begin{equation}\label{Sigma_3(1)-splitting-(2)}
   \Sigma_3^{(1)}\ll \mathscr{L} \cdot\max_{0\leqslant\theta\leqslant1}\Sigma_3^{(2)}(\theta,T).
\end{equation}
For $\Sigma_3^{(2)}(\theta,T)$, we have
\begin{equation*}
         \Sigma_3^{(2)}(\theta,T)
   \ll   \sum_{1\leqslant k\leqslant H}\sum_{MT<h\leqslant4MT}\tau(h)
         \sum_{1\leqslant s\leqslant\frac{4L}{T}}\Bigg|\sum_{\frac{L}{4T}<t\leqslant\frac{4L}{T}}
         e\big(2\alpha h^2stk+\theta t\big)\Bigg|.
\end{equation*}
It follows from Cauchy's inequality that
\begin{align}\label{Sigma_3(2)-square-(3)}
        & \,\,   \Big(\Sigma_3^{(2)}(\theta,T)\Big)^2
     \ll H\Bigg(\sum_{MT<h\leqslant4MT}\tau^2(h)\Bigg)\Bigg(\sum_{1\leqslant s\leqslant\frac{4L}{T}}1\Bigg)
               \sum_{1\leqslant k\leqslant H}\sum_{MT<h\leqslant4MT}
                           \nonumber \\
        & \,\,\qquad\qquad\qquad\qquad\quad\times\sum_{1\leqslant s\leqslant\frac{4L}{T}}
                \Bigg|\sum_{\frac{L}{4T}<t\leqslant\frac{4L}{T}}e\big(2\alpha h^2stk+\theta t\big)\Bigg|^2
                           \nonumber \\
    \ll & \,\, X^\varepsilon HML\sum_{1\leqslant k\leqslant H}\sum_{MT<h\leqslant4MT}
                \sum_{1\leqslant s\leqslant\frac{4L}{T}}
                \sum_{\frac{L}{4T}<t_1,t_2\leqslant\frac{4L}{T}}e\big((2\alpha h^2sk+\theta)(t_1-t_2)\big)
                           \nonumber \\
   \ll & \,\, X^\varepsilon HML\sum_{1\leqslant k\leqslant H}\sum_{1\leqslant s\leqslant\frac{4L}{T}}
               \sum_{\frac{L}{4T}<t_1,t_2\leqslant\frac{4L}{T}}
               \Bigg|\sum_{MT<h\leqslant4MT}e\big(2\alpha h^2sk(t_1-t_2)\big)\Bigg|
                             \nonumber \\
   \ll & \,\,\frac{X^\varepsilon H^2M^2L^3}{T}+X^\varepsilon HML\cdot\Sigma_3^{(3)},
\end{align}
where
\begin{equation*}
     \Sigma_3^{(3)}
   =\sum_{1\leqslant k\leqslant H}\sum_{1\leqslant s\leqslant\frac{4L}{T}}
    \sum_{\substack{\frac{L}{4T}<t_1,t_2\leqslant\frac{4L}{T}\\ t_1\not=t_2}}
    \Bigg|\sum_{MT<h\leqslant4MT}e\big(2\alpha h^2sk(t_1-t_2)\big)\Bigg|.
\end{equation*}
For $\Sigma_3^{(3)}$, we have
\begin{align*}
    \Sigma_3^{(3)}
  = & \,\, \sum_{1\leqslant k\leqslant H}\sum_{1\leqslant s\leqslant\frac{4L}{T}}
           \sum_{\substack{1\leqslant|r_1|\leqslant\frac{4L}{T}\\ 1\leqslant r_2\leqslant\frac{8L}{T}}}
           \Bigg(\sum_{\substack{\frac{L}{4T}<t_1,t_2\leqslant\frac{4L}{T}\\ t_1-t_2=r_1\\ t_1+t_2=r_2}}1\Bigg)
           \Bigg|\sum_{MT<h\leqslant4MT}e\big(2\alpha h^2skr_1\big)\Bigg|
                    \nonumber \\
\ll & \,\, \frac{L}{T}\sum_{1\leqslant k\leqslant H}\sum_{1\leqslant s\leqslant\frac{4L}{T}}
           \sum_{1\leqslant r_1\leqslant\frac{4L}{T}}\Bigg|\sum_{MT<h\leqslant4MT}e\big(2\alpha h^2skr_1\big)\Bigg|
                     \nonumber \\
\ll & \,\, \frac{L}{T}\sum_{1\leqslant n\leqslant\frac{32HL^2}{T^2}}\tau_4(n)
             \Bigg|\sum_{MT<h\leqslant4MT}e\big(\alpha h^2n\big)\Bigg|.
\end{align*}
Therefore, by Cauchy's inequality, one has
\begin{align}\label{sigma_3(3)-square-(4)}
             \Big(\Sigma_3^{(3)}\Big)^2
  \ll & \,\, \frac{L^2}{T^2}\Bigg(\sum_{1\leqslant n\leqslant\frac{32HL^2}{T^2}}\tau_4^2(n)\Bigg)
             \Bigg(\sum_{1\leqslant n\leqslant\frac{32HL^2}{T^2}}\Bigg|\sum_{MT<h\leqslant4MT}
             e\big(\alpha h^2n\big)\Bigg|^2\Bigg)
                     \nonumber \\
 \ll & \,\,\frac{X^\varepsilon HL^4}{T^4}\sum_{1\leqslant n\leqslant\frac{32HL^2}{T^2}}
           \sum_{\substack{MT<h_1,h_2\leqslant 4MT}}e\Big(\alpha\big(h_1^2-h_2^2\big)n\Big)
                     \nonumber \\
 \ll & \,\, \frac{X^\varepsilon H^2ML^6}{T^5}+\frac{X^\varepsilon HL^4}{T^4}\cdot\Sigma_3^{(4)},
\end{align}
where
\begin{equation*}
     \Sigma_3^{(4)}
    =\sum_{\substack{MT<h_1,h_2\leqslant4MT\\ h_1\not=h_2}}
    \Bigg|\sum_{1\leqslant n\leqslant\frac{32HL^2}{T^2}}e\Big(\alpha\big(h_1^2-h_2^2\big)n\Big)\Bigg|.
\end{equation*}
For $\Sigma_3^{(4)}$, by Lemma \ref{expo-Tri-lemma} we have
\begin{align}\label{Sigma_3(4)-upper-1}
           \Sigma_3^{(4)}
  = & \,\, \sum_{1\leqslant|t_1|,|t_2|\leqslant8MT}
            \Bigg(\sum_{\substack{MT<h_1,h_2\leqslant4MT\\ h_1-h_2=t_1\\ h_1+h_2=t_2}}1\Bigg)
            \Bigg|\sum_{1\leqslant n\leqslant\frac{32HL^2}{T^2}}e(\alpha t_1t_2n)\Bigg|
                       \nonumber \\
  \ll & \,\,  \sum_{1\leqslant t_1,t_2\leqslant8MT}
              \Bigg|\sum_{1\leqslant n\leqslant\frac{32HL^2}{T^2}}e(\alpha t_1t_2n)\Bigg|
                        \nonumber \\
 \ll & \,\, \sum_{1\leqslant t_1,t_2\leqslant8MT}\min\bigg(\frac{HL^2}{T^2},\frac{1}{\|\alpha t_1t_2\|}\bigg)
                       \nonumber \\
 \ll & \,\, \sum_{1\leqslant t\leqslant64M^2T^2}\tau(t)\min\bigg(\frac{HL^2}{T^2},\frac{1}{\|\alpha t\|}\bigg).
\end{align}
It follows from Lemma \ref{Vaughan-Lemma-2.2} that
\begin{align}\label{Sigma_3(4)-upper-2}
  & \,\, \sum_{1\leqslant t\leqslant64M^2T^2}\tau(t)\min\bigg(\frac{HL^2}{T^2},\frac{1}{\|\alpha t\|}\bigg)
                       \nonumber \\
  \ll & \,\, X^\varepsilon\sum_{1\leqslant t\leqslant64M^2T^2}
             \min\bigg(\frac{64HM^2L^2}{t},\frac{1}{\|\alpha t\|}\bigg)
                      \nonumber \\
  \ll & \,\, X^\varepsilon HM^2L^2\bigg(\frac{1}{Q}+\frac{T^2}{HL^2}+\frac{Q}{HM^2L^2}\bigg)
                      \nonumber \\
  \ll & \,\, X^\varepsilon \bigg(\frac{HX^2}{Q}+M^2T^2+Q\bigg).
\end{align}
From (\ref{sigma_3(3)-square-(4)}), (\ref{Sigma_3(4)-upper-1}) and (\ref{Sigma_3(4)-upper-2}), we derive that
\begin{equation*}
     \Sigma_3^{(3)}\ll X^\varepsilon
    \Bigg(\frac{HX^{\frac{1}{2}}L^{\frac{5}{2}}}{T^{5/2}}+\frac{HXL^2}{T^2Q^{1/2}}
          +\frac{H^{\frac{1}{2}}XL}{T}
          +\frac{H^{\frac{1}{2}}L^{2}Q^{\frac{1}{2}}}{T^2}\Bigg),
\end{equation*}
which combines (\ref{Sigma_3(2)-square-(3)}) yields
\begin{equation}\label{Sigma_3(2)-last-upper}
     \Sigma_3^{(2)}(\theta,T)\ll X^\varepsilon
    \Bigg(\frac{HXL^{\frac{1}{2}}}{T^{1/2}}+\frac{HX^{\frac{3}{4}}L^{\frac{5}{4}}}{T^{5/4}}
          +\frac{HXL}{TQ^{1/4}}
          +\frac{H^{\frac{3}{4}}X^{\frac{1}{2}}LQ^{\frac{1}{4}}}{T}\Bigg).
\end{equation}
From (\ref{Sigma_3-square-(1)}), (\ref{Sigma_3(1)-splitting-(2)}) and (\ref{Sigma_3(2)-last-upper}), we obtain
\begin{align*}
  \Sigma_3 \ll &\,\, X^\varepsilon \Big(HXL^{-\frac{1}{2}}T^{\frac{1}{2}}+HXL^{-\frac{1}{4}}T^{\frac{1}{4}}
                      +HX^{\frac{7}{8}}L^{\frac{1}{8}}T^{-\frac{1}{8}}+HXQ^{-\frac{1}{8}}
                      +H^{\frac{7}{8}}Q^{\frac{1}{8}}X^{\frac{3}{4}}\Big),
\end{align*}
from which and (\ref{S_I-splitting}), we derive that, under the condition $MT>D_0$, there holds
\begin{align}\label{S_I-upper-case-2}
             S_{I}
   \ll &\,\, X^\varepsilon \max_{1\ll T\ll D}\Bigg(\frac{HXT^{\frac{1}{2}}}{L^{1/2}}
                      +\frac{HXT^{\frac{1}{4}}}{L^{1/4}}
             +\frac{HX^{\frac{7}{8}}L^{\frac{1}{8}}}{T^{1/8}}+\frac{HX}{Q^{1/8}}
             +H^{\frac{7}{8}}Q^{\frac{1}{8}}X^{\frac{3}{4}}\Bigg)
                           \nonumber \\
  \ll &\,\, X^\varepsilon \Bigg(\frac{HXD^{\frac{1}{2}}}{w^{1/2}}
                      +\frac{HXD^{\frac{1}{4}}}{w^{1/4}}
             +\frac{HX^{\frac{7}{8}}L^{\frac{1}{8}}M^{\frac{1}{8}}}{(MT)^{1/8}}+\frac{HX}{Q^{1/8}}
             +H^{\frac{7}{8}}Q^{\frac{1}{8}}X^{\frac{3}{4}}\Bigg)
                           \nonumber \\
  \ll &\,\, X^\varepsilon \Big(HXw^{-\frac{1}{2}}D^{\frac{1}{2}}
                      +HXw^{-\frac{1}{4}}D^{\frac{1}{4}}
             +HXD_0^{-\frac{1}{8}}+HXQ^{-\frac{1}{8}}
             +H^{\frac{7}{8}}Q^{\frac{1}{8}}X^{\frac{3}{4}}\Big).
\end{align}

\subsection{Proof of Lemma \ref{exponential-lemma}}
From (\ref{S_II-upper-last}), (\ref{S_I-upper-case-1}) and (\ref{S_I-upper-case-2}), by taking
\begin{equation}\label{Q-choose}
      Q=X^{\frac{4138}{15}\theta},
\end{equation}
then we deduce that, under the condition (\ref{parameter-1}) and (\ref{parameter-2}), there holds
\begin{equation*}
     S_I\ll X^{1-\varpi}\qquad \textrm{and}\qquad  S_{II}\ll X^{1-\varpi}
\end{equation*}
for some $\varpi>0$. This completes the proof of Lemma \ref{exponential-lemma}.

\section*{Acknowledgement}

The authors would like to express the most sincere gratitude to Professor Wenguang
Zhai for his valuable advice and constant encouragement. Also, the authors appreciate
the referee for his/her patience in refereeing this paper. This work is supported by the National Natural Science
Foundation of China (Grant No. 11901566, 12001047, 11971476, 12071238), the Fundamental Research
Funds for the Central Universities (Grant No. 2019QS02), and the Scientific Research Funds of
Beijing Information Science and Technology University (Grant No. 2025035).

\end{document}